\documentclass[]{article}

\RequirePackage[OT1]{fontenc}
\RequirePackage{amsthm,amsmath,natbib}
\RequirePackage[colorlinks,citecolor=blue,urlcolor=blue]{hyperref}
\usepackage{amsfonts}
\usepackage{graphicx}

\numberwithin{equation}{section}
\theoremstyle{plain}
\newtheorem{theorem}{Theorem}[section]
\newtheorem{lemma}[theorem]{Lemma} 
\newtheorem{corollary}[theorem]{Corollary}
\newtheorem{proposition}[theorem]{Proposition}
\newtheorem{assumption}[theorem]{Assumption}
\newtheorem{definition}[theorem]{Definition}
\newtheorem{proposition-definition}[theorem]{Proposition - Definition}

\theoremstyle{remark}
\newtheorem{remark}{Remark}[section]

\DeclareMathOperator{\R}{{\mathbb{R}}}
\DeclareMathOperator{\N}{{\mathbb{N}}}
\DeclareMathOperator{\supp}{supp}
\DeclareMathOperator{\diam}{diam}
\DeclareMathOperator{\Int}{Int}

\begin{document}

\title{Rate of convergence for geometric inference based on the empirical Christoffel function}

\maketitle

\begin{center}
{ \normalsize Mai Trang Vu$^{*1}$, Fran\c{c}ois Bachoc$^1$, Edouard Pauwels$^2$ }\\
{ \normalsize $1$ Institut de Math\'ematiques de Toulouse, UMR5219. Universit\'e de Toulouse, CNRS. UT3, F-31062 Toulouse, France.} \\
{ \normalsize $2$ Institut de Recherche en Informatique de Toulouse. Universit\'e de Toulouse, CNRS. DEEL, IRT Saint Exupery, Toulouse, France. } \\
{ \normalsize $*$ Corresponding author.} \\
{ \normalsize Email addresses: \url{mai-trang.vu@univ-tlse3.fr}, \url{francois.bachoc@math.univ-toulouse.fr}, \url{edouard.pauwels@irit.fr}  }
\end{center}

\begin{abstract}
We consider the problem of estimating the support of a measure from a finite, independent, sample. The estimators which are considered are constructed based on the empirical Christoffel function. Such estimators have been proposed for the problem of set estimation with heuristic justifications. We carry out a detailed finite sample analysis, that allows us to select the threshold and degree parameters as a function of the sample size. We provide a convergence rate analysis of the resulting support estimation procedure. Our analysis establishes that we may obtain finite sample bounds which are comparable to existing rates for different set estimation procedures. Our results rely on concentration inequalities for the empirical Christoffel function and on estimates of the supremum of the Christoffel-Darboux kernel on sets with smooth boundaries, that can be considered of independent interest.
\end{abstract}


\section{Introduction}

The empirical Christoffel function $\Lambda_{\mu_n,d}$ is defined by an input measure $\mu_n$, which is a scaled counting measure uniformly supported on a cloud of data-points, and by a degree $d \in \N$. It has a strong connection to the population Christoffel function $\Lambda_{\mu,d}$ associated to a measure $\mu$ with density $w$ on an unknown input set $S \subset \R^p$. In particular, typically $\mu_n$ is obtained by a sample from $\mu$, in which case $\Lambda_{\mu_n,d}$ can be seen as an estimation of the population Christoffel function $\Lambda_{\mu,d}$ (see \citet{lasserre2019empirical}).

The (population) Christoffel function $\Lambda_{\mu,d}$ itself has a long history of research in the mathematical analysis literature. Its construction is based on multivariate polynomials of degree at most $d$ and it has strong links to the theory of orthogonal polynomials. Especially, the asymptotic behavior of the Christoffel function as the degree $d$ increases provides useful information regarding the support and density of the associated input measure $\mu$. Important references in multivariate settings include \citet{bos1994asymptotics, bos1998asymptotics, xu1999asymptotics, kroo2013christoffel1, kroo2013christoffel2}, which concern specific cases of the input measure $\mu$ and set $S$. These works not only provide valuable information on the asymptotics of the population Christoffel function as $d$ goes to infinity, but also motivates the usage of this function in statistical contexts, especially in support recovery. Indeed, \citet{lasserre2019empirical} provides a thresholding scheme using the Christoffel function which approximates the compact support $S$ of the measure $\mu$ with strong asymptotic guarantees. More precisely, \citet{lasserre2019empirical} considers a family of polynomial sublevel sets $S_k = \{x \in \R^p: \Lambda_{\mu,d_k}(x) \ge \gamma_k\}$ with $k \in \N$, where the degree $d_k$ increases with $k$ and where the threshold $\gamma_k$ is well-chosen between a lower bound of the Christoffel function inside $S$ and an upper bound outside $S$. Another thresholding scheme can be found in \citet{marx2019tractable}, which provides useful results on the relation between $S$ and its estimator $S_k$. The topic of set estimation based on the population Christoffel function is thus currently a subject of active interest with a large range of applications in machine learning (see \citet{pauwels2016sorting, lasserre2019empirical}).

In a statistical context, the population Christoffel function $\Lambda_{\mu,d}$ is not available and only the empirical Christoffel function $\Lambda_{\mu_n,d}$ is, based on the observed empirical measure $\mu_n$. Let us detail results and discussions presented in \citet{lasserre2019empirical}.
Statistical procedures based on the empirical Christoffel function have three important features: (i) computations are remarkably simple and involve no optimization procedures, (ii) they scale efficiently with the number of observations and (iii) the procedures are affine invariant.
Furthermore, when considering a compactly supported population measure $\mu$ as well as its empirical counterpart $\mu_n$ supported on a sample of $n$ vectors in $\R^p$, drawn independently from $\mu$ and when the degree $d$ is fixed, the empirical Christoffel function $\Lambda_{\mu_n,d}$ converges uniformly to $\Lambda_{\mu,d}$, almost surely with respect to the draw of the random sample. This asymptotic result is appealing 
given the strong connections between $\Lambda_{\mu,d}$ and the support of $\mu$, which suggest that $\Lambda_{\mu_n,d}$ could be used for inferring the support of the population measure $\mu$. Yet more precise quantifications on the relation between sample size $n$ and the degree bound $d$ are required, but \cite{lasserre2019empirical} does not provide any explicit way to choose the degree $d$ as a function of $n$, and does not provide any convergence guaranty for the full plugin approach based on the empirical Christoffel function $\Lambda_{\mu_n,d}$, when $d$ depends on $n$. These shortcomings constitute one of the main motivations for the present work.

\subsection*{Contribution}
Our contribution is twofold:
\begin{enumerate}
	\item We adapt the thresholding scheme in \citet{lasserre2019empirical}, using the empirical Christoffel function, by a careful tuning of the degree $d$ and the threshold level set $\gamma$ in the limit of large sample size. This scheme allows to estimate the compact support $S$ of a measure. Our results include, under regularity assumptions on $\mu$, a quantitative rate of convergence analysis which was unknown for this estimator. More precisely, we consider the Hausdorff distance between the original set $S$ and its estimator $S_n$ and between their respective boundaries, as well as the Lebesgue measure of their symmetric difference. These results rigorously establish the property that, when $n$ is large enough, these distances decrease to zero with an explicit rate.
	\item This analysis relies on results which could be considered of independent interest. First, we provide a quantitative concentration result regarding the convergence of the empirical Christoffel function to its population counterpart. Second, this concentration relies on an estimate of the supremum of the Christoffel Darboux kernel on the support of the underlying measure. We prove that, for a large class of slowly decaying densities with smooth support boundary, this supremum is at most polynomial in the degree $d$. This shows that the considered class of measures is regular in the sense of the Bernstein-Markov property, see \citep{piazzon2016bernstein} and references therein.
\end{enumerate}

\subsection*{Comparison with the existing literature on set estimation}
Support inference (more generally set estimation) has been a topic of research in the statistics literature for more than half a century. The main subject of interest is to infer a set (support of an input measure, level set of an input density function,...) based on samples that are drawn independently from an unknown distribution. Introduction and first results on this subject can be found in \citet{renyi1963konvexe,geffroy1964probleme}, which motivate a subsequent analysis of estimators based on convex hulls for convex domains \citep{chevalier1976estimation} or unions of balls for non-convex sets \citep{devroye1980detection}. More involved estimators follow, such as the excess mass estimator \citep{polonik1995measuring}, the plug-in approach based on the use of density estimators \citep{cuevas1997plug, molchanov1998limit, cuevas2006plug} or the $R$-convex hull of the samples, $R$ being a radius, \citep{rodriguez2007set}.

Those works also motivated the development of minimax statistical analysis for the set estimation problem. We might find minimax results for the recovery of sets with (piecewise) smooth boundaries in \citet{mammen1995asymptotical}, for the estimation of smooth or convex density level sets in \citet{tsybakov1997nonparametric} and for the plug-in approach in \citet{rigollet2009optimal}. More current works related to set estimation include local convex hull estimators \citep{aaron2016local} and cone-convex hulls \citep{cholaquidis2014poincare}.

We obtain convergence rates both in terms of symmetric difference measure, and Hausdorff distance, which can be arbitrarily close to $n^{-1/(p + 2r +2)}$ where $n$ is the sample size, $p$ is the ambient dimension and $r\geq 0$ measures the speed of decrease of the population density around the boundary of the support ($r = 0$ corresponds to a density which is uniformly bounded away from $0$).

Our convergence rates hold under the assumption of a ball of fixed radius $R$ rolling inside and outside the support $S$. The rolling ball assumption is common \citep{cuevas2012statistical,walther1997granulometric,walther1999generalization}. Under this assumption, and in the case $r=0$, \cite{rodriguez2007set} showed that the $R$-convex hull of the samples achieves the rate of convergence $n^{-2/(p+1)}$, for the Hausdorff distance\footnote{Similarly as we do, they consider Hausdorff distances both between sets and between their boundaries.} and the symmetric difference measure.

In \citet{cuevas2004boundary}, the Devroye and Wise estimator is shown to have a convergence rate of order $\left( \log(n)/n\right)^{1/p}$ in Hausdorff distance, under similar geometric assumptions as ours corresponding to the choice $r = 0$. Later on, \citet{biau2008exact} proved for the same estimator, under similar assumptions as ours, a rate which can be arbitrarily close to $n^{-1 / (p + r)}$ for the measure of the symmetric difference for $r = 1$ and $r = 2$. Earlier work presented in \citet{mammen1995asymptotical} proved that $n^{-1/p}$ is minimax optimal for the symmetric difference measure for a special class of piece-wise $C^1$ boundaries. Recently \citet{patschkowski2016adaptation} proved a minimax lower bound on the convergence rate for symmetric difference, of order $n^{-1/(p+r)}$ for adaptive estimators to unknown $r \leq 2$.

Although the rates which we obtain are not optimal, for instance when compared to \cite{rodriguez2007set} in the case $r=0$, the dependency in the dimension and speed of decrease of the density seem reasonable in comparison to existing rates. Let us insist on the fact that our analysis allows to cover a wide range of density decrease regimes and a variety of divergence measures between sets for which the results for other estimates are not known. A detailed comparison between all geometric conditions on the support, its boundary and different notions of divergence between sets is out of reach given the diversity of assumptions in the literature, and as such we only consider a high level general discussion based on orders of magnitude here.

From a computational point of view, our approach using the empirical Christoffel function has important advantages. The most important one is that this approach estimates the support of $\mu$ by a polynomial sublevel set, which is conceptually simple to manipulate. As an important illustration example, consider the situation when one is interested in performing numerical optimization over the estimated support. This situation can arise when a criterion is to be optimized over a feasible domain, which needs to be estimated from data. In this optimization case, the fact the the estimated support is a polynomial sublevel set is beneficial, for instance one can use nonlinear optimization techniques such as Sequential Quadratic Programming (SQP) or barrier functions. If the support is estimated by an union of balls centered at the observations \citep{devroye1980detection}, the estimated support may be less amenable to numerical optimization. Similarly, the $R$-convex hull estimator \citep{rodriguez2007set} is a set over which optimization may be challenging.

In terms of numerical implementation, our approach requires to compute and store the inverse of a matrix of size $s(d_n)  = o(n)$ (see Sections \ref{section_preliminaries} and \ref{section_main_results}) where $d_n$ is the selected degree for the sample size $n$. Then, each input point can be tested to belong to the estimated support or not, with the cost of evaluating a quadratic form of size $s(d_n)$ and of computing $s(d_n)$ monomials in dimension $p$. In practice, $s(d_n)$ is smaller than $n$ (to avoid rank deficiencies), and in our asymptotic results, $d_n$ is selected such that $s(d_n) = o(n)$. Hence our approach relies on reasonably simple and classical computations.
In comparison, for instance, computing the $R$-convex hull estimator \citep{rodriguez2007set} in general dimension $p$ may turn out to be challenging.
In dimension $p=2$, a point can be tested to belong to this set with computational cost $O(n \log n)$ \citep{edelsbrunner1983shape}, but we are not aware of similar efficient algorithms for larger $p$.

\subsection*{Organisation of the paper}
Section \ref{section_preliminaries} introduces the notation and definitions which will be used throughout the text, especially the definition of the population and empirical Christoffel functions and their known properties. In Section \ref{section_main_results}, we present our main assumptions as well as our results on support estimation and convergence of the empirical Christoffel function to the population one. Numerical illustration of the method appears in Section \ref{sec:numerics} for synthetic data in the plane and an outlier detection benchmark in dimension 6. Concluding remarks are provided in Section \ref{section:conclusion}.
The proofs are postponed to the appendix. The appendix also contains additional results of interest on upper and lower bounds on the Christoffel function, outside and inside the support.

\section{Preliminaries} \label{section_preliminaries}
\subsection{General notation}
When $\mu$ is a measure on $\R^p$, we denote by $\supp \mu$ the support of $\mu$. Let $f$ be a measurable function from $\R^p$ to $\R^p$. The push-forward measure of $\mu$ by $f$, denoted by $\mu_{\# f}$, is a measure on $\R^p$ defined by: $\mu_{\# f}(K) = \mu(f^{-1}(K))$ for all Borel sets $K$ of $\R^p$.
Given an arbitrary (measurable) set $S \subset \R^p$, we denote by $\Int S$ the interior of $S$, $\partial S$ the boundary of $S$, $S^c$ the complement of $S$, $\lambda(S)$ the Lebesgue measure of $S$, $\diam (S)$ the diameter of $S$, $\lambda_S$ the Lebesgue measure restricted on $S$ and $\mu_S$ the uniform measure on $S$ (when $\lambda(S)>0$).

When $M$ is a square matrix, we denote by $\|M\|$ the operator norm of $M$, i.e. $$\|M\| = \sup\limits_{x \neq 0} \dfrac{\|Mx\|_2}{\|x\|_2} = \sup\limits_{\|x\|_2=1} \|Mx\|_2.$$
If in addition, $M$ is symmetric and positive definite, we can define its inverse $M^{-1}$ and its unique square root $M^{1/2}$ which are also symmetric positive definite matrices. We denote by $M^{-1/2}$ the inverse of the square root of $M$, which is also symmetric and positive definite.

For $x, y \in \R^p$, we let $d(x,y)$ be the Euclidean norm between $x$ and $y$. For $A \subset \R^p$ and $x \in \R^p$, let $d(x,A) = \inf_{y \in A} d(x,y)$.

We also denote by $B_r(x)$ the open Euclidean ball of radius $r>0$ and centered at $x \in \R^p$ while $\overline{B}_r(x)$ is the associated closed ball. In particular, $B$ denotes the unit Euclidean ball $\overline{B}_1(0)$.

We denote by 
\begin{align}
    \omega_p := \dfrac{2 \pi^{\frac{p+1}{2}}}{\Gamma \left(\frac{p}{2}+1\right)}
     \label{eqn_def_omegap}
\end{align}
the surface area of the $p$-dimensional unit sphere in $\R^{p+1}$. We denote by
\begin{align}
    c_r := \dfrac{\Gamma(p/2+r+1)}{\pi^{p/2} \Gamma(r+1)}
    \label{eqn_def_cr}
\end{align}
the normalization constant of the measure $\nu_r$ whose density is $(1-\|z\|_2^2)^r$ on the unit ball $B$ (see e.g. \citet{xu1999summability}, page 2441, (2.2)). Finally, for $\alpha \in \R$ and $k \in \N$, the associated binomial coefficient is defined as follows:
	$$\binom{\alpha}{k} := \dfrac{\Gamma(\alpha+1)}{\Gamma(k+1) \Gamma(\alpha-k+1)} = \dfrac{\alpha(\alpha-1)(\alpha-2) \ldots (\alpha-k+1)}{k!}.$$
\subsection{Problem setting}
The following notation and assumptions will be standing throughout the text.
\begin{assumption}\hfill
    \begin{enumerate}
        \item $\mu$ is a Borel probability measure on $\R^p$ and its support $S:= \mathrm{supp}(\mu)$ is compact with nonempty interior.
        \item $n \in \N$, $n>0$ is fixed and $X_1,\ldots,X_n$ are independent and identically distributed random vectors  with distribution $\mu$. The corresponding empirical measure is denoted by 
        \begin{align}
            \mu_n = \frac{1}{n} \sum_{i=1}^n \delta_{X_i},
            \label{eqn_def_mun}
        \end{align}
    \end{enumerate}
    where $\delta_x$ is the dirac measure at $x \in \R^p$.
    \label{aspt_statistical_setting}
\end{assumption}
Using the notations of Assumption \ref{aspt_statistical_setting}, given the sample $(X_i)_{i=1}^n$ our goal is to build an estimator $S_n(X_1, \ldots, X_n) \subset \R^p$ in order to approximate $S$. 
We construct a specific kind of estimator $S_n$ based on the empirical Christoffel function. The rest of this section is dedicated to the presentation of further background needed to define our estimator.  Convergence of our estimator to $S$ using different criteria is described next in Section \ref{section_main_results}.

\subsection{The Christoffel function} \label{subsec_Christoffel_function}

\subsubsection*{Multivariate polynomials}
Polynomials of $p$ variables are indexed by the set $\mathbb{N}^p$ of multi-indices.
For example, given a set of $p$ variables $x=(x_1, \ldots, x_p)$ and a multi-index $\alpha=(\alpha_1, \ldots, \alpha_p) \in \N^p$, the \textit{monomial} $x^\alpha$ is given by $x^\alpha = x_1^{\alpha_1}x_2^{\alpha_2} \ldots x_p^{\alpha_p}$ which \textit{degree} is $$\deg x^\alpha = |\alpha| = \sum\limits_{i=1}^p \alpha_i.$$ 
The \textit{space of polynomials of degree at most $d$} is the linear span of monomials of degree up to $d$: $$\Pi_d^p := \text{span}\{x^\alpha: \alpha \in \mathbb{N}^p, |\alpha| \le d\}.$$
The \textit{space of polynomials of $p$ variables} is $$\Pi^p := \bigcup\limits_{d \in \mathbb{N}} \Pi_d^p.$$
The \textit{degree} of a polynomial $P \in \Pi^p$, denoted by $\deg P$, is the maximum degree of its monomial associated to a nonzero coefficient (the null polynomial has degree 0). Note that $\dim \Pi_d^p = \binom{d+p}{d}$. We denote by $s(d)$ the quantity $\binom{d+p}{d}$ throughout the text.

\subsubsection*{Orthonormal polynomials}
Since $\mu$ satisfies Assumption \ref{aspt_statistical_setting}, we have the following inner product: 
$$\left<P, Q \right>_{\mu} = \int\limits_{\R^p} P(x)Q(x)d\mu(x),$$
where $P, Q$ are polynomials. A sequence of \textit{orthonormal polynomials with respect to $\mu$} is a sequence of polynomials $\{P_\alpha: \alpha \in I\}$ in $\Pi^p$ such that $\left<P_\alpha, P_\beta \right>_\mu = \delta(\alpha, \beta)$ \footnote{$\delta(\alpha, \beta)$ is $1$ if $\alpha = \beta$, $0$ otherwise.} for all $\alpha, \beta \in I$. The Gram-Schmidt orthonormalization process guarantees the existence of such an orthonormal sequence. Restricting the degree up to $d \in \N$, we obtain a sequence of orthonormal polynomials $\{P_j: 1 \le j \le s(d)\}$, which is also a basis of $\Pi_d^p$.

\subsubsection*{Moment matrix}
Now, let $\{P_j: 1 \le j \le s(d)\}$ be a basis of $\Pi_d^p$ (not necessarily orthonormal). 
We denote
\begin{eqnarray*}
	v_d: \quad \R^p &\longrightarrow& \R^{s(d)}\\
	\quad x &\longmapsto& (P_1(x), P_2(x), \ldots, P_{s(d)}(x))^T.
\end{eqnarray*}
The \textit{moment matrix  of $\mu$ with respect to the basis $\{P_j\}_{j=1}^{s(d)}$} is a square matrix of dimension $s(d)$ which is defined by
\begin{eqnarray} \label{eqn_moment matrix}
	M_{\mu,d} &=& \int\limits_{\R^p} v_d(x) v_d(x)^T d\mu(x),
\end{eqnarray}
where the integral is taken entry-wise.
We have the following property of the moment matrix which is useful in the sequel.
\begin{lemma} \label{lm_moment_matrix_property}
	Let $P, Q \in \Pi^p_d$ have representations with respect to the basis $\{P_j: 1 \le j \le s(d)\}$ of the form:
	\begin{eqnarray*}
		P = \sum\limits_{j=1}^{s(d)} (c_P)_j P_j = c_P^T v_d, \quad
		Q = \sum\limits_{j=1}^{s(d)} (c_Q)_j P_j = c_Q^T v_d, 
	\end{eqnarray*}
	where $c_P, c_Q \in \R^{s(d)}$.
	Then $$\int\limits_{\R^p} P(x)Q(x)d\mu(x) = c_P^T M_{\mu,d} c_Q.$$
\end{lemma}
\begin{remark} \label{rk_psd_of_moment_matrix}
	$M_{\mu,d}$ is a symmetric, positive definite square matrix of dimension $s(d)$. In fact, for any $c \in \R^{s(d)}$, we set $P_c = \sum\limits_{j=1}^{s(d)} c_j P_j$ and by Lemma \ref{lm_moment_matrix_property}, we have
	\begin{eqnarray} \label{eqn_psd_of_moment_matrix}
	c^T M_{\mu,d}c = \int\limits_{\R^p} P_c(x)^2 d\mu(x) \ge 0.
	\end{eqnarray}
	Since $\Int S \neq \emptyset$, $S$ is polynomial determining, that is, the equality of two polynomials is implied from their equality on the support. By combining this fact with \eqref{eqn_psd_of_moment_matrix}, we obtain that $M_{\mu,d}$ is positive definite.
\end{remark}

\subsubsection*{The Christoffel - Darboux kernel}

The space of polynomials of degree at most $d$ along with the inner product defined by $\mu$ $(\Pi^p_d, \left<.\,,.\right>_\mu)$ is then a finite-dimensional Hilbert space of functions from $\R^p$ to  $\R$ and $\dim \Pi_d^p = s(d)$. Moreover, $(\Pi^p_d, \left<.\,,.\right>_\mu)$ is a reproducing kernel Hilbert space (RKHS)  (see for example \citet{aronszajn1950theory}). Indeed, we notice that the function $P \mapsto P(x)$ is linear on the space of polynomials and $\Pi_d^p$ is finite-dimensional (hence all norms are equivalent), therefore we obtain the continuity of this function on $(\Pi^p_d, \left<.\,,.\right>_\mu)$ for any $x \in \R^p$.
This property of $(\Pi^p_d, \left<.\,,.\right>_\mu)$ guarantees the existence and uniqueness of a reproducing kernel which is defined as follows.
\begin{definition} \label{def_CD_kernel}
	The \textit{Christoffel - Darboux kernel}, denoted by $\kappa_{\mu,d}$, is the reproducing kernel of the RKHS $(\Pi_d^p, \left< . \, ,. \right>_\mu)$, i.e. for all  $x \in \R^p$ and $P \in \Pi_d^p$, we have $\kappa_{\mu,d}(x,.) \in \Pi_d^p$ and
	$$\left< P, \kappa_{\mu,d}(x,.) \right>_\mu = \int\limits_{\R^p} P(y) \kappa_{\mu,d}(x,y)d\mu(y) = P(x).$$
\end{definition}

The two following propositions are explicit formulas for the Christoffel - Darboux kernel. The first one is its expression as a sum of squares of orthonormal polynomials, while the other is a computation based on the moment matrix (and does not require an orthonormal basis).
\begin{proposition} [see e.g \citet{dunkl2014orthogonal}, page 97, (3.6.3)] \label{prop_CD_kernel_SOS}
	Let $\{P_j\}_{j=1}^{s(d)}$be an orthonormal basis of $\Pi_d^p$ with respect to $\mu$. Then for all $x, y \in \R^p$
	$$\kappa_{\mu,d}(x,y) = \sum\limits_{j=1}^{s(d)} P_j(x) P_j(y).$$
\end{proposition}

\begin{proposition} [see e.g. \citet{lasserre2019empirical}, page 7, (3.1)] \label{prop_CD_kernel_moment_matrix}
	Let $v_d = (P_1, P_2, \ldots P_{s(d})^T$ be a basis of $\Pi^p_d$ and $M_{\mu,d}$ be the corresponding moment matrix (see \eqref{eqn_moment matrix}). For all $x,y \in \R^p$, we have
	$$\kappa_{\mu,d}(x,y) = v_d(x)^T M_{\mu,d}^{-1} v_d(y).$$
\end{proposition}

\begin{remark} 
	By Proposition \ref{prop_CD_kernel_SOS}, $$\kappa_{\mu,d}(x,x) = \sum\limits_{j=1}^{s(d)} P_j(x)^2 \ge 0,$$
	where $\{P_j: 1 \le j \le s(d)\}$ is an orthonormal basis of $\Pi_d^p$. Moreover, the $P_j(x)$ cannot be all $0$ since otherwise, the polynomial $1$ will be $0$ at point $x$, which is impossible. So $\kappa_{\mu,d}(x,x) > 0$ for all $x \in \R^p$.
\end{remark}

\subsubsection*{The Christoffel function}
Now, we will define the (population) Christoffel function and provide some of its properties which are useful for the sequel.
\begin{definition} \label{def_Chf_function}
	Let $d \in \mathbb{N}$. The \textit{Christoffel function} associated to $\mu$ and $d$ is the function
	\begin{eqnarray*}
		\Lambda_{\mu,d}: \R^p &\longrightarrow&  \R_+ \\
		\quad z &\longmapsto& \dfrac{1}{\kappa_{\mu,d}(z,z)}.
	\end{eqnarray*} 
\end{definition}

Note that the Christoffel function is well-defined by the positivity of the Christoffel - Darboux kernel. The following proposition is an equivalent definition of the Christoffel function.
\begin{proposition} [see e.g. \citet{dunkl2014orthogonal}, Theorem 3.6.6] \label{prop_Chf_function}
	$$\Lambda_{\mu,d}(z) = \min \left\{ \int\limits_{\R^p} P^2 d\mu: P \in \Pi_d^p, P(z)=1 \right\}.$$
\end{proposition}

We now highlight the following properties of the Christoffel function which will be useful in the sequel. The following proposition guarantees the invariance of the Christoffel function by affine transformations.
\begin{proposition} [see e.g. \citet{pauwels2016sorting}, Lemma 1] \label{prop_Chf_function_affine invariant}
	Let $A$ be an invertible affine map from $\R^p$ to $\R^p$. Recall that $\mu_{\#A}$ is the push-forward measure of $\mu$ by $A$. Then for all $x \in \R^p$,
	$$\Lambda_{\mu,d}(x) = \Lambda_{\mu_{\#A},d}(Ax).$$
\end{proposition}

The next proposition expresses the monotonicity property of the Christoffel function. It is a direct consequence of Proposition \ref{prop_Chf_function}.
\begin{proposition} \label{prop_monotonicity_Ch_function}
	If $\nu$ is a Borel measure on $\R^p$, such that $\nu \le \mu$, in the sense that $\nu(K) \le  \mu(K)$ for all Borel sets $K$, then for all $x \in \R^p$, $$\Lambda_{\nu,d}(x) \le \Lambda_{\mu,d}(x).$$
\end{proposition}
\begin{remark}
All the previous definitions and results extend straightforwardly to the case when $\mu$ does not have unit mass (see Assumption \ref{aspt_statistical_setting}). This extension is a simple scaling. This is a very slight abuse of notations that we will some times do without mentioning it (for instance in Proposition \ref{prop_monotonicity_Ch_function}).
\end{remark}

\subsection{The empirical Christoffel function}

The Christoffel function associated to $\mu_n$ (see Assumption \ref{aspt_statistical_setting}), $\Lambda_{\mu_n, d}$ is called the \textit{empirical Christoffel function}. It is to be compared to the \textit{population Christoffel function} $\Lambda_{\mu, d}$. The convergence of the empirical Christoffel function towards its population counterpart as $n \rightarrow \infty$ and for a fixed $d$ has been shown in \citep{lasserre2019empirical}. This allows by a careful choice of  threshold $\gamma>0$ and degree $d \in \N$, to construct a sequence of polynomial sublevel sets
$$\{x \in \R^p: \Lambda_{\mu_n, d}(x) \ge \gamma\}$$
which estimate the support $S$. It is worth mentioning that the empirical Christoffel function $\Lambda_{\mu_n,d}$ can be computed using the inversion of a square matrix of size $s(d)$ thanks to Proposition \ref{prop_CD_kernel_moment_matrix}. 

\section{Main results} \label{section_main_results}
\subsection{Overview}
From now on, we consider the case where the probability measure $\mu$ has density $w$ with respect to Lebesgue measure.
Our main result is that for a large enough number of observations $n$, by choosing pertinently a  degree $d_n \in \N$ and a threshold $\gamma_n> 0$ for the empirical Christoffel function $\Lambda_{\mu_n,d_n}$, we obtain a sequence of polynomial sublevel sets
\begin{align}
    S_n := \{x \in \R^d: \Lambda_{\mu_n,d_n}(x) \ge \gamma_n \}
    \label{eq:defSn}
\end{align}
which approximates the support of $\mu$.
More explicitly, we show that under smoothness assumptions on $S$, $S_n$ is close to $S$ both in Hausdorff distance and Lebesgue measure of their symmetric difference. For any $\epsilon \in (0,1)$, we obtain an explicit convergence rate of order
\begin{equation} \label{eqn_convergence_rate}
n^{-\frac{1-\epsilon}{p+2r+2}},
\end{equation}
where $r$ measures the speed of decrease of the density of $\mu$, $w$, close to $\partial S$, see Assumption \ref{aspt_density_speed_to_zero}.

Those results are obtained from the following materials:

1. Properties of the population Christoffel function. We provide a lower bound on the Christoffel function $\Lambda_{\mu,d}$ in the interior of the support $S$ and an upper bound in the exterior of $S$. We also provide a bound on the supremum of the Christoffel-Darboux kernel $\kappa_{\mu, d}$ on $S$. Those results will be discussed in Appendices \ref{appen_bounds_Chf_function} and \ref{appen_sup_CD_kernel}.

2. Concentration results for the speed of convergence of the empirical Christoffel function $\Lambda_{\mu_n,d}$ to its population counterpart $\Lambda_{\mu,d}$. This part requires the above mentioned bound on the supremum of the Christoffel - Darboux kernel. Those results could be of independent interest and will be discussed in Subsection \ref{subsec_concentration_result} with all the proofs in Appendix \ref{appen_proof_concentration_result}.

3. We introduce a thresholding scheme using the empirical Christoffel function $\Lambda_{\mu_n,d_n}$ as in \eqref{eq:defSn} by a careful tuning of the degree $d$ and the threshold $\gamma$ in the limit of large sample size $n$. With this thresholding scheme, we prove the desired results described in \eqref{eqn_convergence_rate}. The details will be in Subsection \ref{subsec_main_results_support_estimation} with proofs postponed to Appendix \ref{appen_proofs_main_results_support_estimation}.

\subsection{Conditions on the support and the density} \label{subsec_conditions}
Throughout the text, we consider a probability measure $\mu$ which is supported on $S \subset \R^p$  and has density $w \ge 0$.
\subsubsection{Assumptions on the support $S$}
We first introduce the following definitions, notations and assumptions.
\begin{definition} \label{def_ball_roll_inside_outside}
	Consider a closed set $F \subset \R^p$ and a constant $R>0$. We say that a ball of radius $R$ rolls inside $F$ if for any $x \in F$, there exists a ball $B_x$ centered at $z_x$ of radius $R$ such that $x \in B_x \subset F$. If a ball of radius $R$ rolls inside $\overline{F^c}$, then we say that a ball of radius $R$ rolls outside $F$.
\end{definition}

\begin{definition} \label{def_tube_volume}
    Consider a closed set $F \subset \R^p$. Denote by $F^\epsilon$ the $\epsilon$-extension of $F$, defined as
    $$F^\epsilon = \{x \in \R^p: d(x,F) \le \epsilon\}.$$
    We also define the volume function
    \begin{eqnarray*}
	    V_F: \R_+ &\longrightarrow& \R_+\\
	    \quad \epsilon &\longmapsto& \lambda(F^\epsilon),
    \end{eqnarray*}
    where we recall that $\lambda(.)$ denotes the Lebesgue measure of a set.
\end{definition}
\begin{assumption} \label{aspt_on_S}
	$S \subset \R^p$ is a compact set which has non-empty interior and satisfies:
	\begin{enumerate}
	\item There exists $R > 0$ such that a ball of radius $R$ rolls inside and outside $S$.
	\item For small $\epsilon > 0$, $$V_{\partial S}(\epsilon) \le \epsilon \, C_S  + O(\epsilon^2),$$
	where $C_S > 0$ is a constant which only depends on $S$.
	\end{enumerate}
\end{assumption}
We will rely on Assumption \ref{aspt_on_S} for our results and proofs. The first part of this assumption is made relatively frequently in the support inference literature, see for instance \citet{cuevas2004boundary}.  This part is interpreted as meaning that the boundary of $S$ is smooth. In particular, this assumption prevents corners in the boundary of $S$. The case of sets $S$ with non-smooth boundaries is a future research topic of interest that is not addressed here for the sake of concision. The second part of Assumption \ref{aspt_on_S} will be needed when working with the Lebesgue measure of the symmetric difference.

Next, we provide a class of sets $S$ with some geometric properties, under which Assumption \ref{aspt_on_S} holds.

\begin{lemma} \label{lm_smooth_boundary}
	Let $S \subset \R^p$ be a compact set with non-empty interior such that $\partial S$ is a smooth embedded hypersurface of $\R^p$ (see e.g. \citet{lee2000introduction}, Chapter 5 for definition). Then $S$ satisfies Assumption \ref{aspt_on_S}.
\end{lemma}

The proof of this lemma requires the tubular neighborhood theorem from differentiable geometry (for the first part of Assumption \ref{aspt_on_S}) and Weyl's tube formula (for the second one). The details of the proof are presented in Appendix \ref{appen_proof_smooth_boundary_lemma}. Smoothness of support boundary assumption was considered by \cite{biau2008exact} to analyse the Devroye and Wise estimator.
\subsubsection{Assumption on the density $w$}
Now, for $\delta > 0$, we set
$$L(\delta) := \inf\{w(x): x \in S, d(x, \partial S) \ge \delta\}.$$

The next assumption concerns the rate of decay of the density of $\mu$ at the boundary of the support $S$.

\begin{assumption} \label{aspt_density_speed_to_zero}
	The density $w: S \to \R$ is such that for all $\delta \ge 0$, we have
	$$L(\delta) \ge C \delta^{r},$$
	where $C > 0$ and $r   \ge 0$ are fixed constants (depending only on $\mu$).		
\end{assumption}

\subsection{Main results for support estimation} \label{subsec_main_results_support_estimation}

First, we design our thresholding scheme using the empirical Christoffel function $\Lambda_{\mu_n,d}$. This thresholding scheme depends on the constants $R, C, r$ given by the assumptions on $\mu$ (Assumptions \ref{aspt_on_S} and \ref{aspt_density_speed_to_zero}). It also depends on a constant $\epsilon \in (0,1)$ which can be made arbitrarily small (a smaller $\epsilon$ leads to a better rate of convergence, but possibly worse constants), and on a constant $\alpha \in (0,1)$ which is in principle a small risk threshold, such that our results hold with probability $1 - \alpha$.

Given $0 < \epsilon <1$ and $0 < \alpha < 1$, we define
\begin{equation} \label{eqn_thresholding_scheme_with_d_depending_on_n}
\left\{\begin{aligned}
d_n &:= \left\lfloor \left(\dfrac{C R^{p+r}}{4C_{p,r,\alpha}} n\right)^{\frac{1}{p+2r+2}} \right\rfloor, \\
\gamma_n &:= 12 \left(\dfrac{3p(2-\epsilon)+3(1-\epsilon)r}{2 \epsilon e}\right)^{\frac{p(2-\epsilon) + (1-\epsilon)r}{\epsilon}} \dfrac{1}{d_n^{p(2-\epsilon)+(1-\epsilon)r}}, \\
S_n &:=\{x \in \R^p: \Lambda_{\mu_n, d_n}(x) \ge \gamma_n\},
\end{aligned}\right.
\end{equation}
where 
\begin{eqnarray*}
	C_{p,r, \alpha} &=&  \dfrac{4^{r+2}}{3} \bigg[2^{p+1} c_r \left(\dfrac{e}{p+2r+1}\right)^{p+2r+1} \exp((p+2r+1)^2)\\
	&\,& + \dfrac{4^{p}(p+2)(p+3)(p+8)}{24 \omega_p} \left(\dfrac{e}{p}\right)^p \exp (p^2) \bigg] \left(p+p(1-\log p)+p^2 -\log \alpha\right),
\end{eqnarray*}
$R$ is given in Assumption \ref{aspt_on_S}, $C$, $r$  are given in Assumption \ref{aspt_density_speed_to_zero} and $c_r$, $\omega_p$ are defined in \eqref{eqn_def_cr} and \eqref{eqn_def_omegap} respectively.

The explicit results for this thresholding scheme will be presented in the next subsections. First, we set $$n_0 := \dfrac{4 (D_{p,S,w, \epsilon}+1)^{p+2r+2} C_{p,r,\alpha}}{C R^{p+r}},$$ where

	\begin{eqnarray*}
		D_{p,S, w, \epsilon} &:=& 
		\max\bigg(2, \, \left(\dfrac{\diam(S)}{R}+1\right)^{\frac{1}{1-\epsilon}}, \left(\dfrac{2}{R} E_{p,r,\epsilon}\left(1, \dfrac{1}{2}\right)\right)^{\frac{1}{1-\epsilon}} \bigg),
	\end{eqnarray*}
	with the notation
	\begin{align*}
	&E_{p,r,\epsilon}\left(d, \beta\right) := \\
	& \left(\dfrac{(1+\beta)(p+2)(p+3)(p+8)}{3C(1-\beta) \omega_p}\right)^{\frac{1}{p+r}} \left(\dfrac{3p(2-\epsilon)+3(1-\epsilon)r}{2 \epsilon e}\right)^{\frac{p(2-\epsilon) + (1-\epsilon)r}{\epsilon(p+r)}} \left(\dfrac{e^{1+\frac{p}{d}}}{p}\right)^{\frac{p}{p+r}},
	\end{align*}
	for all $d \in \mathbb{N^*}$ and $\beta \in [0,1[$. Note that $E_{p,r,\epsilon}(d,\beta)$ is a bounded and decreasing function of $d$.
	
We also set $$n_1 := \dfrac{2^{p+2r+4} C_{p,r,\alpha}}{C R^{p+r}},$$
where $C, r$ are given in Assumption \ref{aspt_density_speed_to_zero}, $R$ is given in Assumption \ref{aspt_on_S} and $\omega_p$ is defined in \eqref{eqn_def_omegap}.

\subsubsection{Result for the Hausdorff distance between two sets and two boundaries} \label{subsubsec_Hausdorff_distance}
Recall the definition of the Hausdorff distance between two subsets $A,B$ of $\R^p$:
$$d_H(A,B) = \max\left(\sup\limits_{x\in A} d(x, B), \sup\limits_{y \in B} d(y, A)\right).$$

The following result provides an explicit quantitative rate of convergence for the estimation of $S$ using the thresholding scheme \eqref{eqn_thresholding_scheme_with_d_depending_on_n} based on the empirical Christoffel function. More explicitly, this estimation of $S$ by $S_n$ is measured by the Hausdorff distance between them and between their boundaries. Thus, this theorem is one of the most important results of this paper.
\begin{theorem} \label{thr_support_inference_empirical_Chf_function_Hausdorff_distance}
	Let $S \subset \R^p$ satisfy Assumption \ref{aspt_on_S} - part 1 with radius $R>0$ and $w: S \longrightarrow \R$ satisfy Assumption \ref{aspt_density_speed_to_zero} with two constants $C > 0$ and $r \ge 0$. Let $\mu$ be the measure supported on $S$ with density $w$ with respect to Lebesgue measure. 
	Then, for $n \ge n_0$, the thresholding scheme \eqref{eqn_thresholding_scheme_with_d_depending_on_n} satisfies with probability at least $1-\alpha$ that
	$$d_H(S,S_n) \le \delta_n$$
	and $$d_H(\partial S, \partial S_n) \le \delta_n,$$
	where 
	\begin{eqnarray*}
		\delta_n &:=& \max\bigg(\dfrac{\diam(S)}{d_n^{1-\epsilon}-1}, \dfrac{2}{d_n^{1-\epsilon}} \left(\dfrac{(p+2)(p+3)(p+8)}{C \omega_p}\right)^{\frac{1}{p+r}}\\
		&\,& \times \left(\dfrac{3p(2-\epsilon)+3(1-\epsilon)r}{2 \epsilon e}\right)^{\frac{p(2-\epsilon) + (1-\epsilon)r}{\epsilon(p+r)}} \left(\dfrac{e^{1+\frac{p}{d_n}}}{p}\right)^{\frac{p}{p+r}}\bigg)\\
		&=& O\left(n^{-\frac{1-\epsilon}{p+2r+2}}\right).
	\end{eqnarray*}
\end{theorem} 

By setting 
\begin{equation} \label{eqn_delta1}
    \delta_1(d) := \dfrac{\diam(S)}{d^{1-\epsilon}-1}= O\big(d^{-(1-\epsilon)}\big)
\end{equation}
and 
\begin{equation} \label{eqn_delta2}
    \delta_2(d, \beta) := \dfrac{2}{d^ {1-\epsilon}} E_{p,r,\epsilon}\left(d, \beta\right) = O\big(d^{-(1-\epsilon)}\big),
\end{equation}
we can rewrite $\delta_n$ as:
$$\delta_n = \max\left(\delta_1(d_n), \delta_2\left(d_n, \dfrac{1}{2}\right)\right).$$
These notations will help in the proof of Theorem \ref{thr_support_inference_empirical_Chf_function_Hausdorff_distance} and also in the statements of the next results below.

\subsubsection{Result for the Lebesgue measure of the symmetric difference between two sets} \label{subsubsec_symmetric_distance}
Recall the definition of the symmetric difference between two subsets $A, B$ of $\R^p$: $$A \triangle B = (A \setminus B) \cup (B \setminus A).$$

In this section, in order to measure the convergence of the estimator $S_n$ to the true set $S$, we will use the Lebesgue measure of their symmetric difference: $$(A, B) \longmapsto \lambda(A \triangle B).$$

The following result, which is a counterpart of Theorem \ref{thr_support_inference_empirical_Chf_function_Hausdorff_distance} for the Lebesgue measure of the symmetric difference, is the second main result of this paper.
\begin{theorem} \label{thr_support_inference_empirical_Chf_function_symmetric_distance}
	Let $S \subset \R^p$ satisfy Assumption \ref{aspt_on_S} with radius $R>0$ and $C_S > 0$, $w: S \longrightarrow \R$ satisfy Assumption \ref{aspt_density_speed_to_zero} with two constants $C > 0$ and $r \ge 0$. Let $\mu$ be the measure supported on $S$ with density $w$ with respect to Lebesgue measure. 
	Then, for $n \ge n_1$, the thresholding scheme \eqref{eqn_thresholding_scheme_with_d_depending_on_n} satisfies with probability at least $1-\alpha$ that
	$$\lambda(S \triangle S_n) \le 2 C_S \delta_n + O((\delta_n)^2),$$
	where $\delta_n = O\left(n^{-\frac{1-\epsilon}{p+2r+2}}\right)$ is defined after Theorem \ref{thr_support_inference_empirical_Chf_function_Hausdorff_distance}.
\end{theorem} 
\begin{remark}
	The order of magnitude of the error for the thresholding scheme \eqref{eqn_thresholding_scheme_with_d_depending_on_n} is $n^{-\frac{1-\epsilon}{p+2r+2}}$ for both the Hausdorff distance between two sets and between their boundaries as well as the Lebesgue measure of their symmetric difference. Since $\epsilon \in (0,1)$ can be taken arbitrarily small, the rate of convergence is essentially $n^{-\frac{1}{p+2r+2}}$.
\end{remark}

\begin{remark}
The tuning of $d_n$ and $\gamma_n$ in \eqref{eqn_thresholding_scheme_with_d_depending_on_n} depends on the constants $C$ and $r$ from Assumption \ref{aspt_density_speed_to_zero} and on the constant $R$ from Assumption \ref{aspt_on_S}. 
In practice, these constants are typically unknown and we leave the question of selecting $d_n$ and $\gamma_n$ in a data-driven way, for instance by cross validation, open for future research.

On a theoretical level, the main aim of this paper is to show that it is possible to obtain rates of convergence, by selecting $d_n$ and $\gamma_n$ according to the constants $C$, $r$ and $R$. For the sake of concision, the situation where $C$, $r$ and $R$ are estimated from data is not studied in this paper. Let us nevertheless discuss it briefly here. First, we remark that if Assumptions \ref{aspt_density_speed_to_zero} and \ref{aspt_on_S} hold with constants $C$, $r$ and $R$, then they hold a fortiori with constants $C' < C$, $r' > r$ and $R' < R$. Hence, in order to obtain rates of convergence, it is sufficient to tune $d_n$ and $\gamma_n$ based on conservative values of $C$ and $R$ that are overly small and of $r$ that are overly large, such that Assumptions \ref{aspt_density_speed_to_zero} and \ref{aspt_on_S} hold. Obtaining conservative values is statistically easier than obtaining the sharpest possible values of $C$, $r$ and $R$ such that Assumptions \ref{aspt_density_speed_to_zero} and \ref{aspt_on_S} hold.
Another important question is adaptivity: obtaining a procedure based on the Christoffel function, with no knowledge of the values of $C$, $r$ and $R$ such that Assumptions \ref{aspt_density_speed_to_zero} and \ref{aspt_on_S} hold, and which yields the same rates of convergence as when knowing the sharpest values of $C$, $r$ and $R$ such that Assumptions \ref{aspt_density_speed_to_zero} and \ref{aspt_on_S} hold. 
\end{remark}

\subsubsection{Sketch of the main proofs} \label{subsubsec_connection}
First, we suppose that the estimation of the population Christoffel function by its empirical counterpart can be controlled. More explicitly, we assume that there exists a constant $\beta < 1$ such that for all $x \in \R^p$,
	\begin{eqnarray} \label{eqn_relation_mu_and_mu_n}
	|\Lambda_{\mu,d}(x) - \Lambda_{\mu_n,d}(x)| \le \Lambda_{\mu,d}(x) \beta,
	\end{eqnarray}
or equivalently
    \begin{eqnarray} \label{eqn_relation_mu_and_mu_n_another_expression}
	(1-\beta) \Lambda_{\mu,d}(x) \le \Lambda_{\mu_n,d}(x) \le (1+\beta) \Lambda_{\mu,d}(x).
	\end{eqnarray}
	
Now we introduce a sequence of polynomial sublevel sets which estimates the support $S$ using the empirical function $\Lambda_{\mu_n, d}$ where $d$ does not depend on $n$.

For $0 < \epsilon < 1$ fixed and for $d \in \N$, we define:
    \begin{equation} \label{eqn_thresholding_scheme_with_d_and_n}
    \left\{\begin{aligned}
    \gamma_d &:= 8(1+\beta) \left(\dfrac{3p(2-\epsilon)+3(1-\epsilon)r}{2 \epsilon e}\right)^{\frac{p(2-\epsilon) + (1-\epsilon)r}{\epsilon}} \dfrac{1}{d^{p(2-\epsilon)+(1-\epsilon)r}}, \\
    S_{d,n} &:= \{x \in \R^p: \Lambda_{\mu_n, d}(x) \ge \gamma_d\}.
\end{aligned}\right.
\end{equation}

The idea of this estimator comes from \citet{marx2019tractable}, Section 4.1. The difference is that we let $0<\epsilon<1$ arbitrarily small for a better rate of convergence (instead of setting $\epsilon = 1/2$ like in \citet{marx2019tractable}). Moreover, by choosing carefully the threshold, we obtain an estimator $S_{d,n}$ such that not only $S_{d,n}$ is contained in a small enlargement of $S$ (which has been shown in \citet{marx2019tractable}), but we also have a small enlargement of $S_{d,n}$ that contains $S$. The explicit result is as follows.

\begin{lemma} \label{lm_bounds_of_support}
	Let $S$ be a compact set with non-empty interior, $w: S \longrightarrow \R$ satisfy Assumption \ref{aspt_density_speed_to_zero} with two constants $C > 0$, $r \ge 0$ and $\mu$ be the measure supported on $S$ with density $w$.
	Assume that there exists a constant $\beta < 1$ for which \eqref{eqn_relation_mu_and_mu_n} holds.
	We define
	$$S^1 := \{x \in \R^p: d(x,S) \le \delta_1(d)\}$$
	and $$S^2 := \{x \in S: d(x, \partial S) \ge \delta_2(d, \beta)\},$$
	where $\delta_1(d)$, $\delta_2(d, \beta)$ are defined in \eqref{eqn_delta1} and \eqref{eqn_delta2} respectively.
	Then the thresholding scheme \eqref{eqn_thresholding_scheme_with_d_and_n} satisfies that $$S^2 \subset S_{d,n} \subset S^1.$$
\end{lemma}

This above relation between $S$ and $S_{d,n}$ is important since it implies that the difference between $S$ and $S_{d,n}$ is controlled by a decreasing sequence:
$$\delta_d = \max\big(\delta_1(d), \delta_2(d, \beta)\big) = O \big( d^{-(1-\epsilon)}\big).$$

By adding some assumptions on $S$, we can obtain results concerning the Hausdorff distances and the Lebesgue measure of the symmetric distance between $S$ and $S_{d,n}$.
\begin{enumerate}
    \item For the Hausdorff distances between two sets and between two boundaries, we need Assumption \ref{aspt_on_S} - part 1.
    \item For the Lebesgue measure of their symmetric difference, Assumption \ref{aspt_on_S} - part 2 is required.
\end{enumerate}

Now, under Assumption \ref{aspt_on_S} - part 1 and Assumption \ref{aspt_density_speed_to_zero} and thanks to the concentration results in Subsection \ref{subsec_concentration_result}, we can select $d_n$ such that \eqref{eqn_relation_mu_and_mu_n} holds with high probability with $\beta = 1/2$. Subsequently, we can select a threshold $\gamma_n$ that will optimize the convergence rate of $S_n$ to $S$. We obtain now the thresholding scheme \eqref{eqn_thresholding_scheme_with_d_depending_on_n} and all the results regarding the Hausdorff distances and the Lebesgue measure of the symmetric difference between $S$ and $S_n$ will follow.

All the proofs' details are postponed to Appendix \ref{appen_proofs_main_results_support_estimation} for the sake of clarity.

\subsection{A concentration result for the approximation of the Christoffel function by its empirical counterpart} \label{subsec_concentration_result}
Let $\mu$ be a measure which satisfies Assumption \ref{aspt_statistical_setting} and $\mu_n$ be the corresponding empirical measure. We consider now the speed of convergence of the empirical Christoffel function $\Lambda_{\mu_n,d}$ towards $\Lambda_{\mu,d}$. All the proofs of the following results will be postponed to Appendix \ref{appen_proof_concentration_result}.

First, we state below a technical lemma which bounds uniformly the quantity $ |\Lambda_{\mu_n,d} - \Lambda_{\mu,d}|/\Lambda_{\mu,d}$ by the operator norm of a moment-based random matrix.
\begin{lemma} \label{lm_technical}
	Let $v_d= \{P_j: 1 \le j \le s(d)\}$ be a basis  of orthonormal polynomials with respect to $\mu$. Denote by $M_{\mu_n,d}$ the moment matrix of $\mu_n$ with respect to the basis $v_d$ (see Subsection \ref{subsec_Christoffel_function}).
	Then for all $x \in \mathbb{R}^p$, we have
	$$\big|\Lambda_{\mu,d}(x)- \Lambda_{\mu_n,d}(x)\big| \le \Lambda_{\mu,d}(x) \, \|M_{\mu_n,d}-I_{s(d)}\|$$
	where we recall that the norm of $s(d) \times s(d)$ matrices is the operator norm.
\end{lemma}
Note that $I_{s(d)}$ is actually the associated moment matrix of $\mu$ with respect to the basis $v_d$.
Now, to control the operator norm of the random matrix $M_{\mu_n,d} - I_{s(d)}$, we rely on Theorem 5.44 from \citet{vershynin2010introduction}. The following theorem makes use of this random matrix result and of Lemma \ref{lm_technical} to obtain an upper bound for the quantity $ |\Lambda_{\mu_n,d} - \Lambda_{\mu,d}|/\Lambda_{\mu,d}$ with high probability.

\begin{theorem} \label{thr_concentration_general}
	Let $\mu$ be a measure which satisfies Assumption \ref{aspt_statistical_setting} and $\mu_n$ be the corresponding empirical measure.
	Then for all $x \in \R^p$ and $\alpha > 0$, we have
	$$\big|\Lambda_{\mu,d}(x) - \Lambda_{\mu_n,d}(x)\big| \le \Lambda_{\mu,d}(x) \,\max \left( \sqrt{\dfrac{16m}{3n} \log \dfrac{s(d)}{\alpha}},\dfrac{16m}{3n} \log \dfrac{s(d)}{\alpha} \right) $$
	with probability at least $1 - \alpha$, where $$m = \sup\limits_{x \in \supp \mu}\kappa_{\mu,d}(x,x).$$
\end{theorem}
Note that in our case, the supremum of the Christoffel - Darboux kernel $m$ has a quantitative upper bound of order $O(d^{p+2r+1})$ which is of independent interest and will be provided in Appendix \ref{appen_sup_CD_kernel}. The following corollary is a consequence of Theorem \ref{thr_concentration_general} combined with Theorem \ref{thr_sup_CD_kernel}, and is useful in the tuning of $d_n$ for the thresholding scheme \eqref{eqn_thresholding_scheme_with_d_depending_on_n}.
\begin{corollary} \label{coroll_concentration_result}
	Let $S \subset \R^p$ satisfy Assumption \ref{aspt_on_S} - part 1 with radius $R>0$ and $w: S \longrightarrow \R$ satisfy Assumption \ref{aspt_density_speed_to_zero} with two constants $C>0$ and $r \ge 0$. Let $\mu$ be the measure supported on $S$ with density $w$ with respect to Lebesgue measure and $\mu_n$ be the corresponding empirical measure.
	Then for all $d \ge 2$, $x \in \R^p$ and $\alpha > 0$, we have
	$$|\Lambda_{\mu,d}(x) - \Lambda_{\mu_n,d}(x)| \le \Lambda_{\mu,d}(x)\, \max \left( \sqrt{\dfrac{16\, m(d,p,S,w)}{3n} \log \dfrac{s(d)}{\alpha}},\dfrac{16 \,m(d,p,S,w)}{3n} \log \dfrac{s(d)}{\alpha} \right) $$
	with probability at least $1 - \alpha$, where 	
	\begin{eqnarray*}
		m(d,p,S,w) &=& \dfrac{4^{p+r} s(d)}{C \omega_p R^{p+r}} \dfrac{(d+p+1)(d+p+2)(2d+p+6)}{(d+1)(d+2)(d+3)} \\
		&+& \dfrac{2^{p+2r} c_r}{C R^{p+r}} \left[ 2 \binom{p+d+2r+1}{d} - \binom{p+d+2r}{d} \right].
	\end{eqnarray*}
\end{corollary}
\begin{remark}
	When the dimension $p$ is fixed and $n$ is large, this uniform upper bound in high probability of $|\Lambda_{\mu, d}-\Lambda_{\mu_n, d}|/\Lambda_{\mu, d}$ is of order $\sqrt{d^{p+2r+1}/n}$, up to multiplicative $\log(d)$ factors.
\end{remark}

\section{Numerical illustration}
\label{sec:numerics}
\subsection{Simulated data in the plane}
We consider two synthetic datasets in the plane as depicted in Figure \ref{fig:syntheticData}. The densities considered are uniform on chosen sets with smooth boundary. According to \eqref{eqn_thresholding_scheme_with_d_depending_on_n}, the degree bound should be proportional to $n^{1/4}$. For each value of $n$, we implement the following procedure:
\begin{itemize}
    \item Choose $d_n = \lfloor 2 n^{1/4} \rfloor$.
    \item Evaluate the empirical Christoffel function at each input point.
    \item Choose the smallest value as a threshold and draw the corresponding level set.
\end{itemize}
The results are presented in Figure \ref{fig:syntheticData}. This illustrates the fact that the method is able to identify the support set as well as its boundary and topological features correctly for large enough sample sizes.
\begin{figure}
    \centering
    \includegraphics[width = \textwidth]{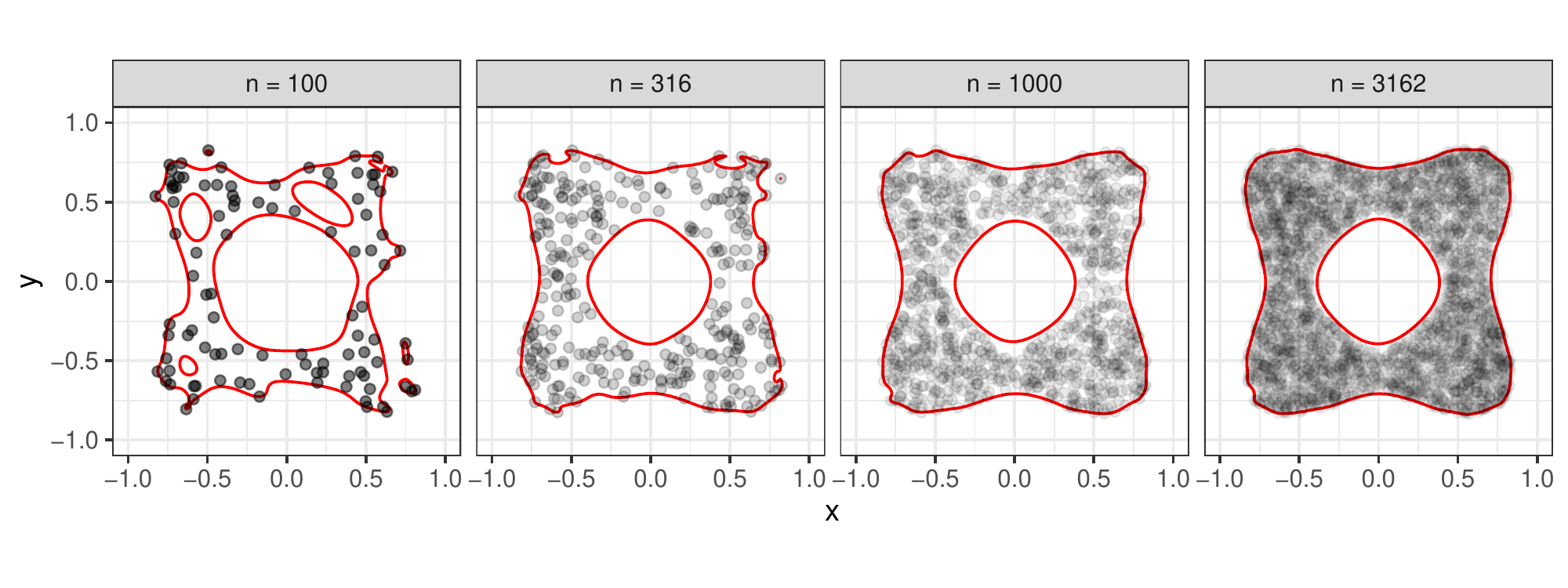}
    \includegraphics[width = \textwidth]{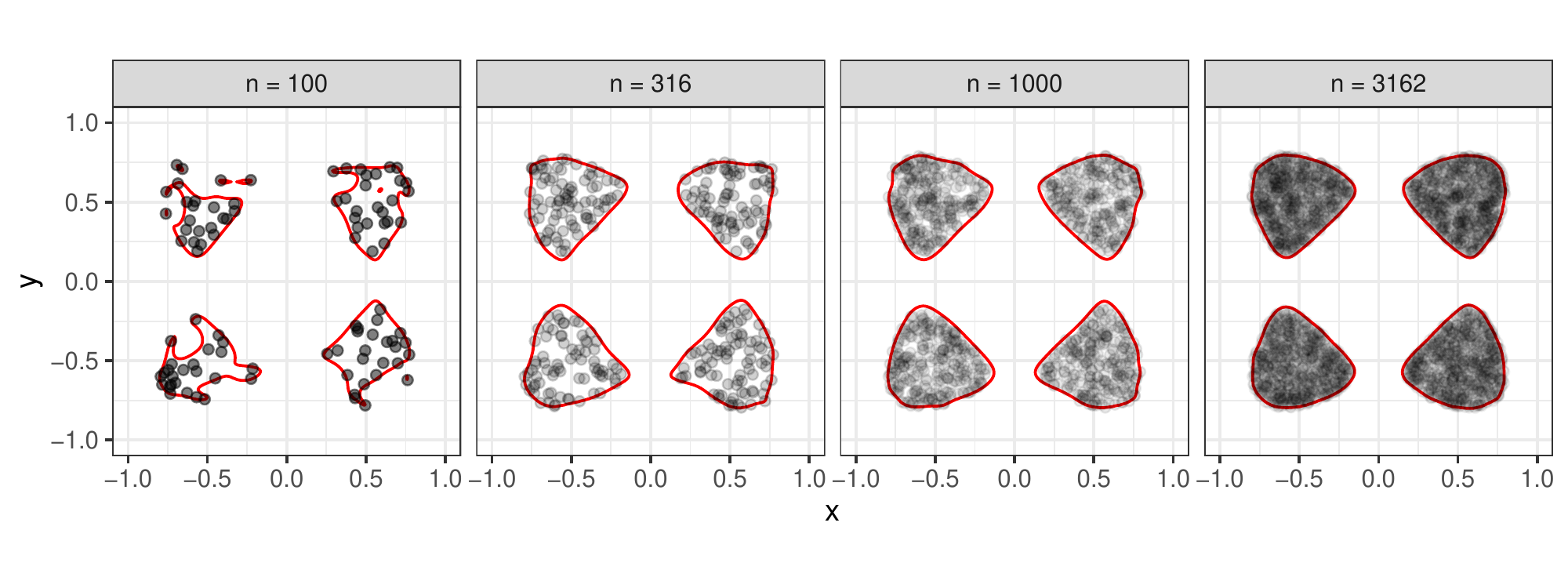}
    \caption{Experiment on synthetic data in the plane. The sample points are drawn uniformly on a chosen set with smooth boundary. We consider two such supports: one with a hole and on with four connected components. The red line displays the boundary of the support estimated from the sample. As we can see, for $n$ large enough, the method is able to identify the support set as well as its boundary and topological features correctly.}
    \label{fig:syntheticData}
\end{figure}

\subsection{Outlier detection}

We consider a thyroid disease dataset obtained from UCI repository \citep{Dua2017UCI}. This is a classification benchmark which contains 3772 examples with three classes, normal, hyperfunctioning and subnormal classes. The hyperfunctioning class contains 93 examples considered as outliers. Each example has 6 numerical descriptors so the effective dimension is $6$. This dataset was used in \cite{aggarwal2015theoretical,keller2012hics} to benchmark outlier detection methods.

We adopt the following procedure, for each concurent method.
\begin{itemize}
    \item Split the dataset randomly into a training set of normal examples and a test set with half malfunctioning cases and half normal examples.
    \item Estimate the support of the training set. This is done by computing a function for which a sublevel set represents the support, for example the Christoffel function or a kernel density estimate, and thresholding to a chosen value to obtain a set.
    \item On the test set predict outlyingness for half of the data for which the estimated function is most below the chosen threshold value on the support. Not that in this case, the threshold value is not important, only the order and rank of function value on the test set matters.
\end{itemize}
The results are displayed in Figure \ref{fig:thyroid}. We compare the Christoffel function with varying degree to kernel density estimators using Laplace or Gaussian kernels with various bandwidth. These results suggest that on this benchmark, the Christoffel function performs favorably and is more stable with respect to the choice of tuning parameter compared to kernel density estimators.

\begin{figure}
    \centering
    \includegraphics[width = \textwidth]{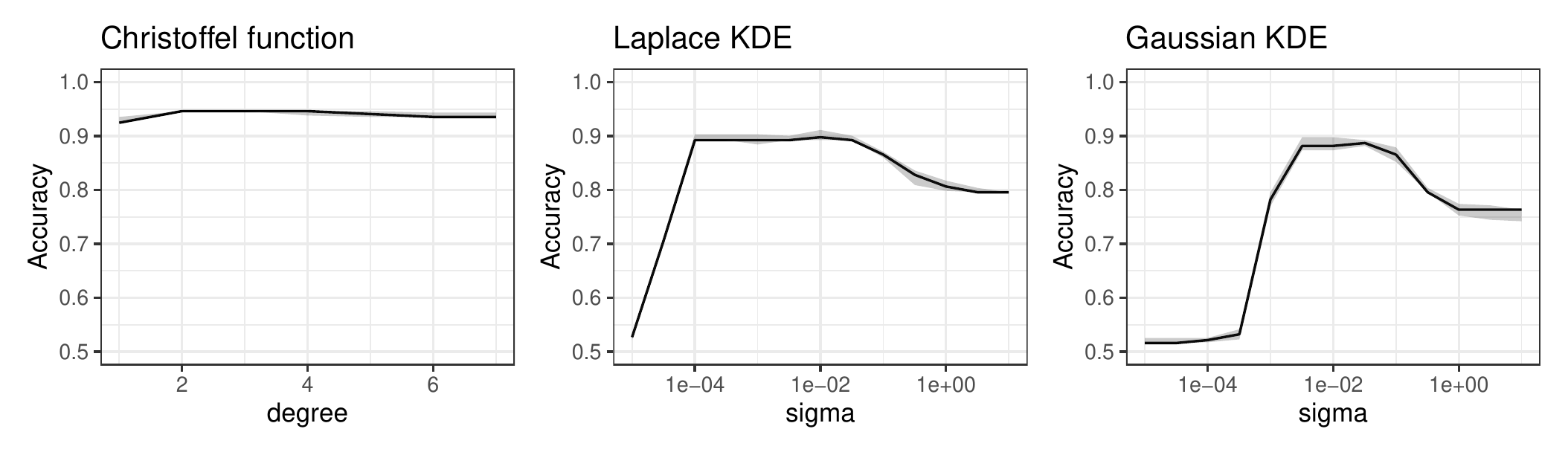}
    \caption{Experiment on the thyroid dataset. The dataset and experiments are described in the main text. We compare Christoffel function and kernel density estimation (KDE) to detect malfunctioning cases in a test set, based on a training set containing only normal cases. For each method, the middle line shows the median and the ribbons show quantiles for 10 random train test splits.}
    \label{fig:thyroid}
\end{figure}

\section{Conclusion} \label{section:conclusion}

We have provided a detailed quantitative finite sample analysis of
support estimation based on the empirical Christoffel function. We have
obtained a sample-size-dependent choice of the degree $d$, together with
errors bounds for the corresponding support inference procedure. An
interest of our results is that support inference based on
the empirical Christoffel function is computationally and conceptually
attractive, as we illustrate in Section \ref{sec:numerics}. These procedures
have recently been subject to active developments, but
there are only weak theoretical guarantees.

Our error rates are, generally speaking, comparable to convergence rates obtained by concurent support inference methods. 
Differences in rates relate to the fact that our
proofs are based on tools and developments from different
fields, in particular matrix concentration inequalities, non-parametric
statistics, geometry and orthonormal polynomials. Furthermore, our
setting is quite general, in terms of assumptions on the unknown support
and on the divergences between sets. In future work,
it would be interesting to see if our proofs could be refined to obtain
slightly sharper bounds in more specific settings. Alternatively,
it would be interesting to see if lower bounds can be provided
specifically for estimation procedures based on the empirical
Christoffel function, paving the way to a minimax theory for this approach.

Other problems of interest remain open. In particular, it would be
interesting to extend our results to the case of supports with
non-smooth boundaries. It would also be valuable to provide a
quantitative analysis of the case where the underlying measure is
supported on a manifold with smaller dimension than the ambient space.

\section*{Acknowledgements}
This work was mostly carried out during the first author's MS final project which was funded by the CIMI labex (Centre International de Math\'ematiques et d'Informatique de Toulouse). The third author acknowledges support of ANR-3IA Artificial and Natural Intelligence Toulouse Institute, ANR MASDOL and  Air Force Office of Scientific Research, Air Force Material Command, USAF, under grant number  FA9550-19-1-7026/19IOE033 and FA9550-18-1-0226. The authors would like to thank Jonas Kahn and S\'ebastien Gerchinovitz for motivating preliminary discussions on the topic. The authors are also grateful to Vincent Guedj, for his advices on Lemma \ref{lm_smooth_boundary} and its proof.

\bibliographystyle{apalike}
\bibliography{ref}
\pagebreak
\appendix
\section{Bounds on the Christoffel function} \label{appen_bounds_Chf_function}
The following results provide a lower bound on the Christoffel function $\Lambda_{\mu, d}$ inside the support $S$ and an upper bound outside $S$. These bounds are similar to those in Subsections 6.3.1 and 6.3.2 of \citet{lasserre2019empirical} and will be useful in the next proofs.

\subsection{Upper bound on the Christoffel function outside $S$}
In this section, we consider a probility measure $\mu$ which satisfies Assumption \ref{aspt_statistical_setting}.
Now, to exhibit an upper bound on the Christoffel function outside $S$, we first provide a refinement of the ``needle polynomial" which has been introduced in \citet{kroo2013christoffel1}. 
\begin{lemma} [see e.g. \cite{lasserre2019empirical}, Lemma 6.3] \label{lm_needle_polynomial}
	For any $d \in \N^*$ and $\delta \in (0,1)$, there exists a $p$-variate polynomial $Q$ of degree $2d$ such that $Q(0) = 1$, $|Q| \le 1$ on the unit ball $B$ and $|Q| \le 2^{1-\delta d}$ on $B \setminus \overline{B}_\delta(0)$. 
\end{lemma}
\begin{lemma} \label{lm_upper_bound_outside}
	Let $\delta > 0$ and $x \notin S$ such that $d(x, S) \ge \delta$. Then, for any $d \in \N^*$ we have
	$$\Lambda_{\mu, d}(x) \le 2^{3- \frac{\delta d}{\delta+ \diam(S)}}.$$
\end{lemma}
\begin{proof}
	First, we will prove Lemma \ref{lm_upper_bound_outside} with $x \notin S$ such that $d(x,S) = \delta$. In this case, $S \subset T:= \overline{B}_{\delta+ \diam (S)}(x) \setminus B_\delta(x)$. Indeed, for any $y \in S$, $d(x,y) \le d(x, S)+ \diam (S) = \delta + \diam(S)$. On the other hand, if $y \in B_\delta(x)$, then $d(x,S) \le d(x,y) < \delta$ which is a contradiction.\\
	Now, let $A$ be the affine transformation which maps $\overline{B}_{\delta+ \diam (S)}(x)$ to the unit ball $B$ and $\mu_{\#A}$ be the push-forward measure of $\mu$ by $A$. Then $\supp \mu_{\#A} = A(S) \subset A(T) = B \setminus B_{\delta'}(0)$ where $\delta' = \frac{\delta}{\delta+ \diam(S)}$ and by Proposition \ref{prop_Chf_function_affine invariant}, we have
	$$\Lambda_{\mu,d}(x) = \Lambda_{\mu_{\#A},d}(0).$$
	Next, we apply Lemma \ref{lm_needle_polynomial} to $k \in \N^*$ and $\delta' \in (0,1)$, we obtain a polynomial $Q$ of degree $2k$ such that $Q(0)=1$ and $|Q| \le 2^{1-\delta'k}$ on $B \setminus \overline{B}_{\delta'}(0)$, which implies that $|Q| \le 2^{1-\delta'k}$ on $\supp \mu_{\#A}$. Thus
	\begin{eqnarray*} 
		\Lambda_{\mu_{\# A},2k}(0) &=& \min \left\{ \int\limits_{\mathbb{R}^p} P^2 d\mu_{\# A}: P \in \Pi_{2k}^p, P(z)=1 \right\}\\
		&\le& \int\limits_{\mathbb{R}^p} Q^2 d\mu_{\#A} = \int\limits_{\supp \mu_{\#A}} Q^2 d\mu_{\#A}\\
		&\le& 2^{2(1-\delta'k)} \le 2^{3-\delta'2k}.
	\end{eqnarray*}
	Then we have for any $k \in \N^*$
	$$\Lambda_{\mu, 2k}(x) \le 2^{3- \delta'2k}.$$
	Now, the equivalent definition of the Christoffel function in Proposition \ref{prop_Chf_function} makes sure that
	$$\Lambda_{\mu, 2k+1}(x) \le \Lambda_{\mu_S, 2k}(x) \le 2^{2(1-\delta'k)} \le 2^ {2(1-\delta'k)+1-\delta'} \le 2^{3-\delta'(2k+1)}.$$
	By combining both cases $d=2k$ and $d=2k+1$, we have
	$$\Lambda_{\mu, d}(x) \le 2^{3-\delta'd}.$$
	Finally, since $2^{3-\delta'd} = 2^{3-\frac{\delta d}{\delta+\diam(S)}}$ is a decreasing function of $\delta$, we have for all $x$ such that $d(x,S) \ge \delta$,
	$$\Lambda_{\mu, d}(x) \le 2^{3- \frac{d(x,S) d}{d(x,S)+\diam(S)}} \le 2^{3-\frac{\delta d}{\delta+\diam(S)}}.$$
\end{proof}
\subsection{Lower bound on the Christoffel function inside $S$}
We now consider a compact set $S$ with non-empty interior, a density $w$ satisfying Assumption \ref{aspt_density_speed_to_zero} with two constants $C > 0$, $r \ge 0$ and the measure $\mu$ supported on $S$ with density $w$.
\begin{lemma} \label{lm_lower_bound_inside}
	Let $\delta > 0$ and $x \in S$ such that $d(x, \partial S) \ge \delta$. Then for any $d \ge 2$ we have
	$$\Lambda_{\mu, d}(x) \ge \dfrac{C \omega_p \, \delta^{p+r}}{2^{p+r}} \dfrac{1}{s(d)} \dfrac{(d+1)(d+2)(d+3)}{(d+p+1)(d+p+2)(2d+p+6)}.$$
\end{lemma}
\begin{proof}
	First, we will prove that the closed ball $\overline{B}_{\delta/2}(x) \subset \{x \in S: d(x, \partial S) \ge \delta/2\} \subset S$. Indeed, if $z \in \overline{B}_{\delta/2}(x)$, i.e. $\mathrm{dist}(x,z) \le \delta/2$, then $$d(z, \partial S) \ge d(x, \partial S) - d(x,z) \ge \delta - \delta/2 = \delta/2.$$
	We have
	\begin{eqnarray*}
		\Lambda_{\mu,d}(x) &=& \min \left\{ \int\limits_{\R^p} P^2(z) d\mu(z): P \in \Pi_d^p, P(x) = 1 \right\} \\
		&=& \min \left\{ \int\limits_{S} P^2(z)w(z) dz: P \in \Pi_d^p, P(x) = 1 \right\} \\
		&\ge& \min \left\{ \int\limits_{\overline{B}_{\delta/2}(x)} P^2(z) d\mu(z): P \in \Pi_d^p, P(x) = 1 \right\} \\
		&\ge& L \left(\dfrac{\delta}{2}\right) \min \left\{ \int\limits_{\overline{B}_{\delta/2}(x)} P^2(z) dz: P \in \Pi_d^p, P(x) = 1 \right\} \\
		&\ge& C \left(\dfrac{\delta}{2}\right)^r \min \left\{ \int\limits_{\R^p} P^2(z) d\lambda_{\overline{B}_{\delta/2}(x)}(z): P \in \Pi_d^p, P(x) = 1 \right\} \\
		&=& C \left(\dfrac{\delta}{2}\right)^r \Lambda_{\lambda_{\overline{B}_{\delta/2}(x)},d}(x),
	\end{eqnarray*}
	where the third inequality comes from Assumption \ref{aspt_density_speed_to_zero}.
	Now we have
	\begin{eqnarray*}
		\Lambda_{\lambda_{\overline{B}_{\delta/2}(x)},d}(x) &=& \lambda\left(\overline{B}_{\delta/2}(x)\right) \Lambda_{\mu_{\overline{B}_{\delta/2}(x)},d}(x) = \lambda\left(\overline{B}_{\delta/2}(x)\right) \Lambda_{\mu_B,d}(0)\\
		&=& \dfrac{\lambda\left(\overline{B}_{\delta/2}(x)\right)}{\lambda(B)} \Lambda_{\mu_B,d}(0) = \left(\dfrac{\delta}{2}\right)^p \Lambda_{\mu_B,d}(0)\\
		&\ge& \left(\dfrac{\delta}{2}\right)^p \dfrac{\omega_p}{s(d)} \dfrac{(d+1)(d+2)(d+3)}{(d+p+1)(d+p+2)(2d+p+6)},
	\end{eqnarray*}
	where the first and third equality come from the monotonicity of the Christoffel function, the second equality comes from its affine invariance and the last inequality is Lemma 6.1 in \citet{lasserre2019empirical}, which is obtained when $d \ge 2$. Then, by combining the above arguments, we have the lower bound result.
\end{proof}

\section{Supremum of the Christoffel - Darboux kernel on $S$} 
\label{appen_sup_CD_kernel}
In this section, we consider a set $S$ satisfying Assumption \ref{aspt_on_S} - part 1 with radius $R>0$ and a density $w$ satisfying Assumption \ref{aspt_density_speed_to_zero} with two constants $C>0$, $r \ge 0$. Let $\mu$ be the measure supported on $S$ with density $w$. First, we have the following upper bound of the Christoffel - Darboux kernel $\kappa_{\mu,d}$ inside the support $S$, which is a direct consequence of Lemma \ref{lm_lower_bound_inside}.
\begin{corollary} \label{coroll_sup_CD_kernel_inside}
	Let us consider $x \in S$ such that $d(x, \partial S) \ge R/2$. Then
	$$\kappa_{\mu,d}(x,x) \le \dfrac{4^{p+r} s(d)}{C \omega_p R^{p+r}} \dfrac{(d+p+1)(d+p+2)(2d+p+6)}{(d+1)(d+2)(d+3)}.$$
\end{corollary}
\begin{proof}
	We apply Lemma \ref{lm_lower_bound_inside} with $\delta = R/2$ and we use the fact that $\kappa_{\mu,d}(x,x) = 1/\Lambda_{\mu, d}(x)$.
\end{proof}
Next, for the points which stay near the boundary of $S$, we will rely on Theorem 3.1 from \citet{xu1999summability} which provides an explicit formula for the Christoffel - Darboux kernel associated to a measure with Jacobi-like weight on the unit Euclidean ball. The following lemma provides an upper bound near the boundary.
\begin{lemma} \label{lm_sup_CD_kernel_near_boundary}
	Given $x \in S$ such that $d(x, \partial S) \le R/2$, we have
	$$\kappa_{\mu,d}(x,x) \le \dfrac{2^{p+2r} c_r}{C R^{p+r}} \left[ 2 \binom{p+d+2r+1}{d} - \binom{p+d+2r}{d} \right],$$
	where $c_r$ is defined in \eqref{eqn_def_cr}.
\end{lemma}
\begin{proof}
	By Assumption \ref{aspt_on_S} - part 1, there exists a point $z_x \in S$ such that $x \in \overline{B}_R(z_x) \subset S$. We set $\epsilon = \|x-z_x\|_2$, then $\epsilon \le R$ and
	$$\epsilon \ge d(z_x, \partial S) - d(x, \partial S) \ge R - R/2 = R/2.$$
	Evidently, $x$ is on the boundary of the closed ball $\overline{B}_\epsilon(z_x) \subset S$. Moreover, for all $y \in \overline{B}_{\epsilon}(z_x)$, we have
	$$d(y, \partial S) \ge d(y, \partial \overline{B}_\epsilon(z_x)) = \epsilon - \|y-z_x\|_2 \ge 0.$$
	Now, by Assumption \ref{aspt_density_speed_to_zero}, we have for all $y \in \overline{B}_\epsilon(z_x)$:
	\begin{eqnarray*}
		w(y) \ge C \left(\epsilon - \|y-z_x\|_2\right)^r &\ge& C \left(\epsilon - \|y-z_x\|_2\right)^r \left(\dfrac{\epsilon + \|y-z_x\|_2}{2 \epsilon}\right)^r\\
		&=& \dfrac{C}{(2\epsilon)^r} \left(\epsilon^2 - \|y-z_x\|_2^2\right)^r.
	\end{eqnarray*}
	We have
	\begin{eqnarray*}
		\Lambda_{\mu, d}(x) &=& \min \left\{ \int\limits_{\R^p} P^2(y) d\mu(y): P \in \Pi_d^p, P(x) = 1 \right\} \\
		&=& \min \left\{ \int\limits_{S} P^2(y) w(y) dy: P \in \Pi_d^p, P(x) = 1 \right\} \\
		&\ge& \min \left\{ \int\limits_{\overline{B}_\epsilon(z_x)} P^2(y) w(y) dy: P \in \Pi_d^p, P(x) = 1 \right\} \\
		&\ge& \dfrac{C}{(2\epsilon)^r} \min \left\{ \int\limits_{\overline{B}_\epsilon(z_x)} P^2(y) \left(\epsilon^2 - \|y-z_x\|_2^2\right)^r dy: P \in \Pi_d^p, P(x) = 1 \right\}.
	\end{eqnarray*}
	Now, by changing the variable $z = \dfrac{y-z_x}{\epsilon}$ and setting $Q(z) = P(z_x+\epsilon z)$, we have
	\begin{eqnarray*}
		\int\limits_{\overline{B}_\epsilon(z_x)} P^2(y) \left(\epsilon^2 - \|y-z_x\|_2^2\right)^r dy &=& \int\limits_{B} P^2(z_x + \epsilon z) (\epsilon^2- \epsilon^2 \|z\|_2^2)^r \epsilon^p dz \\
		&=& \epsilon^{p+2r} \int\limits_B Q^2(z) (1- \|z\|_2^2)^r dz, 
	\end{eqnarray*}
	where $Q$ is a polynomial with degree at most $d$ and $Q \left(\frac{x-z_x}{\epsilon}\right) = P(x) = 1$. We set $\tilde{x} = \frac{x-z_x}{\epsilon} \in \partial B$ since $x \in \partial \overline{B}_\epsilon(z_x)$. Now we have
	\begin{eqnarray*}
		\Lambda_{\mu, d}(x) &\ge& \dfrac{C \epsilon^{p+2r}}{(2\epsilon)^r} \min \left\{\int\limits_B Q^2(z) (1- \|z\|_2^2)^r dz: Q \in \Pi_d^p, Q(\tilde{x})= 1 \right\}\\
		&=& \dfrac{C \epsilon^{p+r}}{2^r c_r} \min \left\{\int\limits_B Q^2(z) \, c_r (1- \|z\|_2^2)^r dz: Q \in \Pi_d^p, Q(\tilde{x})= 1 \right\}\\
		&=& \dfrac{C \epsilon^{p+r}}{2^r c_r} \Lambda_{\nu_r, d} (\tilde{x}),
	\end{eqnarray*}
	where we recall that $c_r = \dfrac{\Gamma(p/2+r+1)}{\pi^{p/2} \Gamma(r+1)}$ is the normalization constant of the measure $\nu_r$ which density is $(1-\|z\|_2^2)^r$ on the unit ball $B$. Then, by taking the inverse, we have
	$$\kappa_{\mu, d}(x,x) \le \dfrac{2^r c_r}{C \epsilon^{p+r}} \kappa_{\nu_r, d}(\tilde{x}, \tilde{x}).$$
	Now, to compute $\kappa_{\nu_r,d}(\tilde{x},\tilde{x})$, we use Theorem 3.1 in \citet{xu1999summability} with $\mu= r + \frac{1}{2}$ and we obtain that
	\begin{eqnarray*}
		\kappa_{\nu_r,d}(\tilde{x}, \tilde{x}) &=& \sum\limits_{k=0}^d \dfrac{k+ r + \frac{1}{2} + \frac{p-1}{2}}{r+\frac{1}{2} + \frac{p-1}{2}} \int\limits_0^\pi C_k^{\left(r + \frac{1}{2} + \frac{p-1}{2}\right)}\left(\left<\tilde{x},\tilde{x}\right>+ \sqrt{1-\|\tilde{x}\|_2^2} \sqrt{1-\|\tilde{x}\|_2^2} \cos \psi\right)\\
		&\times& (\sin \psi)^{2\left(r+\frac{1}{2}\right)-1} d\psi \bigg/ \int\limits_0^\pi (\sin \psi)^{2\left(r+\frac{1}{2}\right)-1} d\psi\\
		&=& \sum\limits_{k=0}^d \dfrac{k+\frac{p}{2}+r}{ \frac{p}{2}+r} C_k^{\left(\frac{p}{2}+r\right)}(1),
	\end{eqnarray*}
	where the $C_k^{(\beta)}$'s are the classical Gegenbauer polynomials, which are orthogonal polynomials on $[-1,1]$ with respect to the weight function $(1-x^2)^{\beta-1/2}$. In particular, by \citet[p.81, (4.7.3)]{szeg1939orthogonal}, we have $$C_k^{\left(\frac{p}{2}+r\right)}(1) = \binom{p+k+2r-1}{k},$$
	where we recall that the binomial coefficient for $\alpha \in \R$ and $k \in \N$ is defined as:
	$$\binom{\alpha}{k} := \dfrac{\Gamma(\alpha+1)}{\Gamma(k+1) \Gamma(\alpha-k+1)} = \dfrac{\alpha(\alpha-1)(\alpha-2) \ldots (\alpha-k+1)}{k!}.$$
	Hence
	\begin{eqnarray*}
		\kappa_{\nu_r,d}(\tilde{x},\tilde{x}) &=& \sum\limits_{k=0}^d \dfrac{k+ \frac{p}{2}+r}{ \frac{p}{2}+r} C_k^{\left(\frac{p}{2}+r\right)}(1) = \sum\limits_{k=0}^d \dfrac{k+ \frac{p}{2}+r}{ \frac{p}{2}+r} \binom{p+k+2r-1}{k}\\
		&=& \sum\limits_{k=0}^d \dfrac{2k+p+2r}{p+2r} \dfrac{(p+k+2r-1)(p+k+2r-2) \ldots (p+2r)}{k!}\\
		&=& \sum\limits_{k=0}^d \left( \dfrac{2(k+p+2r)}{p+2r}-1 \right) \dfrac{(p+k+2r-1)(p+k+2r-2) \ldots (p+2r)}{k!}\\
		&=& 2 \sum\limits_{k=0}^d \dfrac{(p+k+2r)(p+k+2r-1) \ldots (p+2r+1)}{k!} - \binom{p+k+2r-1}{k}\\
		&=& 2 \sum\limits_{k=0}^d \binom{p+k+2r}{k} - \sum\limits_{k=0}^d \binom{p+k+2r-1}{k}\\
		&=& 2 \binom{p+d+2r+1}{d} - \binom{p+d+2r}{d}.
	\end{eqnarray*}
	We finally have
	$$\kappa_{\mu, d}(x,x) \le \dfrac{2^r c_r}{C \epsilon^{p+r}} \left[2 \binom{p+d+2r+1}{d} - \binom{p+d+2r}{d}\right],$$
	and the result follows by using the fact that $\epsilon \ge R/2$.
\end{proof}
By combining Corollary \ref{coroll_sup_CD_kernel_inside} and Lemma \ref{lm_sup_CD_kernel_near_boundary}, we have the following theorem regarding the supremum of the Christoffel - Darboux kernel.
\begin{theorem} \label{thr_sup_CD_kernel}
	Let $S \subset \R^p$ satisfies Assumption \ref{aspt_on_S} - part 1 with radius $R>0$ and $w: S \longrightarrow \R$ satisfies Assumption \ref{aspt_density_speed_to_zero} with two constants $C>0$ and $r \ge 0$. Let $\mu$ be the measure supported on $S$ with density $w$ with respect to Lebesgue measure. We have for any $d \ge 2$,
	\begin{eqnarray*}
		\sup\limits_{x \in S} \kappa_{\mu,d}(x,x) &\le& \dfrac{4^{p+r} s(d)}{C \omega_p R^{p+r}} \dfrac{(d+p+1)(d+p+2)(2d+p+6)}{(d+1)(d+2)(d+3)} \\
		&+& \dfrac{2^{p+2r} c_r}{C R^{p+r}} \left[ 2 \binom{p+d+2r+1}{d} - \binom{p+d+2r}{d} \right].
	\end{eqnarray*}
\end{theorem}
We denote the above upper bound by $m(d,p,S,w)$. Note that this bound is of order $d^{p+2r+1}$ when $p$ is fixed.
\section{Proof of the concentration results} \label{appen_proof_concentration_result}

\begin{proof} [Proof of Lemma \ref{lm_technical}]
	For all $x \in \R^p$, we have
	\begin{eqnarray*}
		\big|\Lambda_{\mu,d}(x) - \Lambda_{\mu_n,d}(x)\big| &=& \Lambda_{\mu,d}(x) \Lambda_{\mu_n,d}(x) \, \big|\kappa_{\mu,d}(x,x)-\kappa_{\mu_n,d}(x,x)\big|\\
		&=& \Lambda_{\mu,d}(x) \Lambda_{\mu_n,d}(x) \, \Big|v_d(x)^T \left(I_{s(d)}- M_{\mu_n,d}^{-1} \right) v_d(x) \Big|\\
		&=& \Lambda_{\mu,d}(x) \Lambda_{\mu_n,d}(x) \, \Big|v_d(x)^T \left(M_{\mu_n,d}^{-1/2}\right)^T \left(M_{\mu_n,d}- I_{s(d)} \right) M_{\mu_n,d}^{-1/2} v_d(x) \Big|\\
		&=& \Lambda_{\mu,d}(x) \Lambda_{\mu_n,d}(x) \, \Big|\left(M_{\mu_n,d}^{-1/2} v_d(x)\right)^T \left(M_{\mu_n,d}- I_{s(d)} \right) M_{\mu_n,d}^{-1/2} v_d(x) \Big|\\
		&\le& \Lambda_{\mu,d}(x) \Lambda_{\mu_n,d}(x) \left(M_{\mu_n,d}^{-1/2} v_d(x)\right)^T M_{\mu_n,d}^{-1/2} v_d(x) \, \|M_{\mu_n,d}- I_{s(d)}\|\\
		&=& \Lambda_{\mu,d}(x) \Lambda_{\mu_n,d}(x) v_d(x)^T M_{\mu_n,d}^{-1} v_d(x) \, \|M_{\mu_n,d}- I_{s(d)}\|\\
		&=& \Lambda_{\mu,d}(x) \Lambda_{\mu_n,d}(x) \kappa_{\mu_n,d}(x,x) \, \|M_{\mu_n,d}- I_{s(d)}\|\\
		&=& \Lambda_{\mu,d}(x) \, \|M_{\mu_n,d}- I_{s(d)}\|,
	\end{eqnarray*}
	where the third equality comes from the fact that $M_{\mu_n,d}$ is symmetric and positive definite, which implies that $M_{\mu_n,d}^{-1/2}$ exists and is also symmetric; while the inequality can be seen as
	$$\big|z^T A z\big| \le z^T z \|A\|,$$
	with $z = M_{\mu_n,d}^{-1/2} v_d(x) \in \R^{s(d)}$ and $A = M_{\mu_n,d}- I_{s(d)} \in \R^{s(d) \times s(d)}$.
	This inequality can be proved as below:
	$$\big|z^TAz\big| = \left<z,Az \right> \le \|z\|_2 \, \|Az\|_2 \le \|z\|_2 \|A\|\, \|z\|_2 = z^Tz \|A\|,$$
	where the first inequality is Cauchy-Schwarz and the second one comes from the definition of operator norm.
\end{proof}

\begin{proof} [Proof of Theorem \ref{thr_concentration_general}]
	Let $v_d = \{P_j: 1 \le j \le s(d)\}$ be a system of orthonormal polynomials with respect to $\mu$ and $M_{\mu_n,d}$ be the moment matrix of $\mu_n$ with respect to $v_d$. We apply Theorem 5.44 in \citet{vershynin2010introduction} to $$A = \left[ {\begin{array}{ccc}
		P_1(X_1) & \ldots & P_{s(d)}(X_1) \\
		\ldots &  & \ldots \\
		P_1(X_n) & \ldots & P_{s(d)}(X_n)
		\end{array} } \right],$$
	which is a $n \times s(d)$ random matrix whose rows $A_k = \left(P_1(X_k), \ldots, P_{s(d)}(X_k)\right)$ are independent random vectors in $\mathbb{R}^{s(d)}$ with the common second moment matrix $\Sigma = \mathbb{E}[A_k^T A_k] = I_{s(d)}$. We have $$ \dfrac{1}{n} A^T A = M_{\mu_n,d},$$ thus $$\bigg\| \dfrac{1}{n} A^T A - \Sigma \bigg\| = \|M_{\mu_n,d}-I_{s(d)}\|.$$
	If we can obtain an almost sure bound on the rows of $A_k$, then Theorem 5.44 in \citet{vershynin2010introduction} provides an upper bound for $\|M_{\mu_n}-I_{s(d)}\|$, and then, by Lemma \ref{lm_technical}, an upper bound for $|\Lambda_{\mu,d}(x)- \Lambda_{\mu_n,d}(x)|$ with high probability.
    Let us check the boundedness condition of the rows $A_k$. We have
	$$\|A_k\|_2^2 = \sum\limits_{j=1}^{s(d)} P_j(X_k)^2 = \kappa_{\mu,d}(X_k,X_k) = \dfrac{1}{\Lambda_{\mu,d}(X_k)}.$$
	A natural upper bound for this will be $$\sup\limits_{x \in \supp \mu} \sum\limits_{j=1}^{s(d)} P_j(x)^2 = \sup\limits_{x \in \supp \mu}\kappa_{\mu,d}(x,x) :=m, $$
	which is finite since $x \mapsto \sum_{j=1}^{s(d)} P_j(x)^2$ is continuous and $\supp \mu$ is a compact set.
	Now, by Theorem 5.44 in \citet{vershynin2010introduction}, for all $t \ge 0$, with probability at least $1 - s(d). \exp(-3t^2/16)$, we have
	$$\|M_{\mu_n,d}- I_{s(d)}\| \le \max \bigg(t \sqrt{ \dfrac{m}{n}}, \dfrac{t^2m}{n}\bigg).$$
	We choose $\alpha = s(d). \exp(-3t^2/16)$, which means $t= \sqrt{\dfrac{16}{3} \log \dfrac{s(d)}{\alpha}}$, and we have
	$$\|M_{\mu_n,d}- I_{s(d)}\| \le \max \left( \sqrt{\dfrac{16m}{3n} \log \dfrac{s(d)}{\alpha}},\dfrac{16m}{3n} \log \dfrac{s(d)}{\alpha} \right)$$
	with probability at least $1 - \alpha$.	
	Then by lemma \ref{lm_technical}, $$\big|\Lambda_{\mu,d}(x) - \Lambda_{\mu_n,d}(x)\big| \le \Lambda_{\mu,d}(x)\, \max \left( \sqrt{\dfrac{16m}{3n} \log \dfrac{s(d)}{\alpha}},\dfrac{16m}{3n} \log \dfrac{s(d)}{\alpha} \right) $$
	with probability at least $1 - \alpha$.	
\end{proof}

\section{Proofs of the main results regarding support estimation} \label{appen_proofs_main_results_support_estimation}
\subsection{Proof of Lemma \ref{lm_bounds_of_support}} \label{appen_proofs_connection}
First, we introduce some inequalities which will be useful in the proof of Lemma \ref{lm_bounds_of_support}.

\begin{lemma} [see e.g. \citet{lasserre2019empirical}, Lemma 6.5] \label{lm_inequality_binom}
	For any $m, n \in \mathbb{N}^*$, we have
	\begin{eqnarray*} 
	\binom{m+n}{m} \le m^n \left(\dfrac{e}{n}\right)^n \exp{\left(\dfrac{n^2}{m}\right)}.
	\end{eqnarray*}
\end{lemma}

\begin{lemma} [see e.g \citet{marx2019tractable}, Lemma 5] \label{lm_inequality_1}
	For any $q>0$, we have
	$$\min\limits_{x>0} \left[\log(2) x-2q \log(x)\right] = 2q \left(1-\log\left(\dfrac{2q}{\log(2)}\right)\right) \ge 2q(1-\log(3q)).$$
\end{lemma}

\begin{lemma} \label{lm_inequality_2}
	For any $d \in \N$, $0 < \epsilon < 1$ and $q>0$, we have
	$$2^{3-d^\epsilon} \le \dfrac{8 (3q)^{2q}}{e^{2q} d^{2q \epsilon}}.$$
\end{lemma}
\begin{proof}
	We have
	\begin{eqnarray*}
		&& 2^{3-d^\epsilon} \le \dfrac{8 (3q)^{2q}}{e^{2q} d^{2q\epsilon}}\\
		&\Leftrightarrow& (3-d^\epsilon) \log(2) \le 3 \log(2) + 2q \log(3q) - 2q - 2q\epsilon \log(d)\\
		&\Leftrightarrow& 2q(1-\log(3q)) \le \log(2) d^\epsilon - 2q \log(d^\epsilon),
	\end{eqnarray*}
	which holds true by applying Lemma \ref{lm_inequality_1} to $x = d^\epsilon>0$.
\end{proof}

\begin{proof} [Proof of Lemma \ref{lm_bounds_of_support}] 
    The inclusion $S_{d,n} \subset S^1$ can be proved for a more general thresholding scheme, where we define for any $d \in \N$ and any $q \in \N$ such that $2q \epsilon>p$:
    \begin{equation} \label{eqn_thresholding_scheme_with_d_general}
        \left\{\begin{aligned}
        \gamma_d &:= \dfrac{8 (1+\beta) (3q)^{2q}}{e^{2q} d^{2q \epsilon}} \\
        S_{d,n} &:= \{x \in \R^p: \Lambda_{\mu_n, d}(x) \ge \gamma_d\}.
        \end{aligned}\right.
\end{equation} 

	Indeed, if $x \notin S^1$, i.e. $d(x, S) > \delta_1(d) = \dfrac{\diam(S)}{d^{1-\epsilon}-1}$, we apply Lemma \ref{lm_upper_bound_outside}, then Lemma \ref{lm_inequality_2} and we have
	$$\Lambda_{\mu,d}(x) < 2^{3- \frac{\delta_1(d) d}{\delta_1(d)+ \diam(S)}} = 2^{3-d^\epsilon} \le \dfrac{8 (3q)^{2q}}{e^{2q} d^{2q \epsilon}}.$$
	Since $\Lambda_{\mu_n, d}(x) \le (1+\beta) \Lambda_{\mu,d}(x)$ by \eqref{eqn_relation_mu_and_mu_n_another_expression}, we obtain that
	$$\Lambda_{\mu_n,d}(x) \le \dfrac{8(1+\beta) (3q)^{2q}}{e^{2q} d^{2q\epsilon}} = \gamma_d,$$
	which means that $x \notin S_{d,n}$ and we can deduce the result by contraposition.
	
	By choosing $q = \frac{p(2-\epsilon)+(1-\epsilon)r}{2 \epsilon}$ in the scheme \eqref{eqn_thresholding_scheme_with_d_general}, we obtain our thresholding scheme \eqref{eqn_thresholding_scheme_with_d_and_n} and the result $S_{d,n} \subset S^1$ follows.
	
    Now, if $x \in S^2$, i.e. $x \in S$ and $d(x, \partial S) \ge 
	\delta_2(d, \beta)$, then by Lemma \ref{lm_lower_bound_inside}, we have
	\begin{eqnarray*}
		\Lambda_{\mu,d}(x) &\ge& \dfrac{C \omega_p \, (\delta_d^2)^{p+r}}{2^{p+r}} \dfrac{1}{s(d)} \dfrac{(d+1)(d+2)(d+3)}{(d+p+1)(d+p+2)(2d+p+6)}\\
		&\ge& \dfrac{C \omega_p (\delta_2(d, \beta))^{p+r}}{2^{p+r} d^p \left(\dfrac{e}{p} \right)^p \exp\left(\dfrac{p^2}{d}\right)} \dfrac{24}{(p+2)(p+3)(p+8)}\\
		&=& \dfrac{C \omega_p}{2^{p+r} d^p \left(\dfrac{e}{p} \right)^p \exp\left(\dfrac{p^2}{d}\right)} \dfrac{24}{(p+2)(p+3)(p+8)} \dfrac{2^{p+r}}{d^{(1-\epsilon)(p+r)}} \\
		&\times&  \dfrac{(1+\beta)(p+2)(p+3)(p+8)}{3C(1-\beta) \omega_p} \left(\dfrac{3p(2-\epsilon)+3(1-\epsilon)r}{2 \epsilon e}\right)^{\frac{p(2-\epsilon) + (1-\epsilon)r}{\epsilon}} \left(\dfrac{e^{1+\frac{p}{d}}}{p}\right)^{p}\\
		&=& \frac{8(1+\beta)}{1-\beta} \left(\dfrac{3p(2-\epsilon)+3(1-\epsilon)r}{2 \epsilon e}\right)^{\frac{p(2-\epsilon) + (1-\epsilon)r}{\epsilon}} \dfrac{1}{d^{p(2-\epsilon)+(1-\epsilon)r}}.
	\end{eqnarray*}
	Since $\Lambda_{\mu_n,d}(x) \ge (1-\beta) \Lambda_{\mu,d}(x)$ by \eqref{eqn_relation_mu_and_mu_n_another_expression}, we have $\Lambda_{\mu_n,d}(x) \ge \gamma_d$, which means that $x \in S_{d,n}$.
\end{proof}
\subsection{Proof of Theorem \ref{thr_support_inference_empirical_Chf_function_Hausdorff_distance}} \label{appen_proofs_Hausdorff_distance}
First, we have the following lemma which highlights an important property of a set $S$ which satisfies Assumption \ref{aspt_on_S} - part 1.

\begin{lemma} \label{lm_Hausdorff_distance}
	Let $S \subset \R^p$ satisfies Assumption \ref{aspt_on_S} - part 1 with radius $R > 0$. Given $\delta_1, \delta_2 > 0$, we set
	$$S^1 := \{x \in \R^p: d(x, S) \le \delta_1 \}$$
	and $$S^2 := \{x \in S: d(x, \partial S) \ge \delta_2\}.$$
	Suppose that there exists a closed set $\tilde{S} \subset \R^p$ such that 
	\begin{eqnarray} \label{eqn_bounds_of_support}
	S^2 \subset \tilde{S} \subset S^1.
	\end{eqnarray}
	If we have in addition that $\delta := \max(\delta_1, \delta_2) \le R$, then 
	$$d_H(S, \tilde{S}) \le \delta$$
	and $$d_H(\partial S, \partial \tilde{S}) \le \delta.$$
\end{lemma}
\begin{proof} We will begin with the Hausdorff distance between two sets $S$ and $\tilde{S}$:
	$$d_H(S, \tilde{S}) = \max \left( \sup\limits_{x \in \tilde{S}} d(x,S), \sup\limits_{x \in S} d(x, \tilde{S})\right).$$
	Given $x \in \tilde{S}$, then by \eqref{eqn_bounds_of_support}, $x \in S^1$, i.e. $d(x, S) \le \delta_1 \le \delta$. Hence
	\begin{eqnarray} \label{eqn_Hausdorff_sets_1}
	\sup\limits_{x \in \tilde{S}} d(x, S) \le \delta.
	\end{eqnarray}
	Given now $x \in S$. Since $S^2 \subset \tilde{S}$, we have
	$d(x, \tilde{S}) \le d(x, S^2).$ By Assumption \ref{aspt_on_S} - part 1 and since $\delta_2 \le R$, there exists $z_x \in S$ such that $x \in \overline{B}_{\delta_2}(z_x) \subset S$. Hence $\|x - z_x\|_2 \le \delta_2$ and $d(z_x, \partial S) \ge \delta_2$. We have now $z_x \in S^2$ and
	$d(x, S^2) \le \|x - z_x\|_2 \le \delta_2.$
	Then $d(x, \tilde{S}) \le \delta_2 \le \delta$ and
	\begin{eqnarray} \label{eqn_Hausdorff_sets_2}
	\sup\limits_{x \in S} d(x, \tilde{S}) \le \delta.
	\end{eqnarray}
	By combining \eqref{eqn_Hausdorff_sets_1} and \eqref{eqn_Hausdorff_sets_2}, we have $d_H(S, \tilde{S}) \le \delta$. Now, we continue with the Hausdorff distance between two boundaries:
	$$d_H(\partial S, \partial \tilde{S}) = \max \left(\sup\limits_{x \in \partial \tilde{S}} d(x, \partial S) , \sup\limits_{x \in \partial S} d(x, \partial \tilde{S})\right).$$
	Consider $x \in \partial \tilde{S}$. We will consider two cases where $x \in S$ and $x \notin S$ separately. If $x \notin S$, then $d(x, \partial S) = d(x, S) \le \delta$
	by \eqref{eqn_Hausdorff_sets_1} since $x \in \partial \tilde{S} \subset \tilde{S}$. Now, consider $x \in S \cap \partial \tilde{S}$. Note that since $S^2 \subset \tilde{S}$, $\Int S^2 \subset \Int \tilde{S}$, which implies that 
	$\Int S^2 \cap \partial \tilde{S} = \emptyset.$
	Now when $x \in \partial \tilde{S}$, we have $x \notin \Int S^2 = \{x \in S: d(x, \partial S) > \delta_2\}$. By combining with the fact that $x \in S$, we have $d(x, \partial S) \le \delta_2 \le \delta$. Hence
	\begin{eqnarray} \label{eqn_Hausdorff_boundaries_1}
	\sup\limits_{x \in \partial \tilde{S}} d(x, \partial S) \le \delta.
	\end{eqnarray}
	Consider now $x \in \partial S$. We also consider two cases where $x \in \tilde{S}$ and $x \notin \tilde{S}$ separately. If $x \notin \tilde{S}$, then $d(x, \partial \tilde{S}) = d(x, \tilde{S}) \le \delta$ by \eqref{eqn_Hausdorff_sets_2} since $x \in \partial S \subset S$. Given now $x \in \tilde{S} \cap \partial S$. Note that since $\tilde{S} \subset S^1$ and $x \in \tilde{S}$, we have
	$d(x, \partial \tilde{S}) \le d(x, \partial S^1) = d(x, \overline{(S^1)^c}).$
	When $x \in \partial S$, $x \in \overline{S^c}$. By Assumption \ref{aspt_on_S} - part 1 and since $\delta_1 \le R$, there exists $y_x \in \overline{S^c}$ such that $x \in \overline{B}_{\delta_1}(y_x) \subset \overline{S^c}$. Hence $\|x - y_x\|_2 \le \delta_1$ and $d(y_x, \partial S) = d(y_x, \partial \overline{S^c}) \ge \delta_1$. We have now $y_x \in \overline{(S^1)^c}$ and $d(x, \overline{(S^1)^c}) \le \|x-y_x\|_2 \le \delta_1$. Then $d(x, \partial \tilde{S}) \le \delta_1 \le \delta$. Now we have
	\begin{eqnarray} \label{eqn_Hausdorff_boundaries_2}
	\sup\limits_{x \in \partial S} d(x, \partial \tilde{S}) \le \delta.
	\end{eqnarray}
	By combining \eqref{eqn_Hausdorff_boundaries_1} and \eqref{eqn_Hausdorff_boundaries_2}, we obtain that $d_H(\partial S, \partial \tilde{S}) \le \delta$.
\end{proof}
 Now, by combining the bounds on $S$ which have been shown in Lemma \ref{lm_bounds_of_support} with the previous property of $S$, we obtain a result concerning the Hausdorff distance between two sets and two boundaries for the thresholding scheme \eqref{eqn_thresholding_scheme_with_d_and_n}.
\begin{lemma} \label{lm_Hausdorff_distance_population}
	Under the assumptions and definitions of Lemma \ref{lm_bounds_of_support}, we suppose in addition that $S$ satisfies Assumption \ref{aspt_on_S} - part 1 with radius $R>0$. For any $d > 1$ large enough such that $\delta_d \le R$, the thresholding scheme \eqref{eqn_thresholding_scheme_with_d_and_n} satisfies that
	$$d_H(S, S_{d,n}) \le \delta_d$$
	and $$d_H(\partial S, \partial S_{d,n}) \le \delta_d.$$
\end{lemma}
\begin{proof}
	We use Lemma \ref{lm_bounds_of_support}, then apply Lemma \ref{lm_Hausdorff_distance} with $\tilde{S} = S_{d,n}$, $\delta_1 = \delta_1(d)$, $ \delta_2 =\delta_2(d, \beta)$ under the assumption that $\delta_d = \max\big(\delta_1(d), \delta_2(d, \beta)\big) \le R$.
\end{proof}

The proof of Theorem \ref{thr_support_inference_empirical_Chf_function_Hausdorff_distance} follows by combining Lemma \ref{lm_Hausdorff_distance_population} with the concentration result in Corollary \ref{coroll_concentration_result}.

\begin{proof} [Proof of Theorem \ref{thr_support_inference_empirical_Chf_function_Hausdorff_distance}]
	By Corollary \ref{coroll_concentration_result}, we have with probability at least $1-\alpha$,
	\begin{eqnarray*}
	|\Lambda_{\mu,d_n}(x) - \Lambda_{\mu_n,d_n}(x)| &\le& \Lambda_{\mu,d_n}(x)\, \max \bigg( \sqrt{\dfrac{16\, m(d_n,p,S,w)}{3n} \log \dfrac{s(d_n)}{\alpha}},\\
	&\,&\dfrac{16 \,m(d_n,p,S,w)}{3n} \log \dfrac{s(d_n)}{\alpha} \bigg),
	\end{eqnarray*}
	
	where 	
	\begin{eqnarray*}
		m(d_n,p,S,w) &=& \dfrac{4^{p+r} s(d_n)}{C \omega_p R^{p+r}} \dfrac{(d_n+p+1)(d_n+p+2)(2d_n+p+6)}{(d_n+1)(d_n+2)(d_n+3)} \\
		&+& \dfrac{2^{p+2r} c_r}{C R^{p+r}} \left[ 2 \binom{p+d_n+2r+1}{p+2r+1} - \binom{p+d_n+2r}{p+2r} \right].
	\end{eqnarray*}
	Note that
	\begin{eqnarray*}
		&\,& \dfrac{2^{p+2r} c_r}{C R^{p+r}} \left[ 2 \binom{p+d_n+2r+1}{p+2r+1} - \binom{p+d_n+2r}{p+2r} \right] \le \dfrac{2^{p+2r+1} c_r}{C R^{p+r}} \binom{p+d_n+2r+1}{p+2r+1}\\
		&\le& \dfrac{2^{p+2r+1} c_r}{C R^{p+r}} d_n^{p+2r+1} \left(\dfrac{e}{p+2r+1}\right)^{p+2r+1} \exp((p+2r+1)^2)
	\end{eqnarray*}
	and
	\begin{eqnarray*}
		&\,& \dfrac{4^{p+r} s(d_n)}{C \omega_p R^{p+r}} \dfrac{(d_n+p+1)(d_n+p+2)(2d_n+p+6)}{(d_n+1)(d_n+2)(d_n+3)}\\ &\le& \dfrac{4^{p+r}}{C \omega_p R^{p+r}} d_n^p \left(\dfrac{e}{p}\right)^p \exp (p^2) \dfrac{(p+2)(p+3)(p+8)}{24}\\
		&\le& \dfrac{4^{p+r}}{C \omega_p R^{p+r}} d_n^{p+2r+1} \left(\dfrac{e}{p}\right)^p \exp (p^2) \dfrac{(p+2)(p+3)(p+8)}{24},
	\end{eqnarray*}
	where the inequalities come from Lemma \ref{lm_inequality_binom} and the fact that the function $$d \mapsto \dfrac{(d+p+1)(d+p+2)(2d+p+6)}{(d+1)(d+2)(d+3)}$$ is decreasing.
	Hence
	\begin{eqnarray*}
		\dfrac{16 \,m(d_n,p,S,w)}{3n}  
		&\le& \dfrac{4^{r+2}}{3n C R^{p+r}} d_n^{p+2r+1} \bigg[2^{p+1} c_r \left(\dfrac{e}{p+2r+1}\right)^{p+2r+1} \exp((p+2r+1)^2)\\
		&+& \dfrac{4^{p}(p+2)(p+3)(p+8)}{24 \omega_p} \left(\dfrac{e}{p}\right)^p \exp (p^2) \bigg].
	\end{eqnarray*}
	On the other hand,
	\begin{eqnarray*}
		\log \dfrac{s(d_n)}{\alpha} &\le& \log \left[d_n^p \left(\dfrac{e}{p}\right)^p \exp \left(\dfrac{p^2}{d_n}\right)\right] - \log \alpha = p \log d_n + p(1-\log p) + \dfrac{p^2}{d_n} - \log \alpha\\
		&\le& p d_n + p(1-\log p) + p^2 - \log \alpha\\
		&\le& d_n \left(p+p(1-\log p)+p^2 -\log \alpha\right).
	\end{eqnarray*}
	Then
	$$\dfrac{16 \,m(d_n,p,S,w)}{3n} \log \dfrac{s(d_n)}{\alpha} \le \dfrac{d_n^{p+2r+2}}{n C R^{p+r}}  C_{p,r,\alpha},$$
	where
	\begin{eqnarray*}
		C_{p,r, \alpha} &=&  \dfrac{4^{r+2}}{3} \bigg[2^{p+1} c_r \left(\dfrac{e}{p+2r+1}\right)^{p+2r+1} \exp((p+2r+1)^2)\\
		&\,& + \dfrac{4^{p}(p+2)(p+3)(p+8)}{24 \omega_p} \left(\dfrac{e}{p}\right)^p \exp (p^2) \bigg] \left(p+p(1-\log p)+p^2 -\log \alpha\right).
	\end{eqnarray*}
	Since $d_n = \left\lfloor \left(\dfrac{C R^{p+r}}{4C_{p,r,\alpha}} n\right)^{\frac{1}{p+2r+2}} \right\rfloor$, then $d_n \le \left(\dfrac{C R^{p+r}}{4C_{p,r,\alpha}} n\right)^{\frac{1}{p+2r+2}}$, which implies that
	$$\dfrac{d_n^{p+2r+2}}{n C R^{p+r}}  C_{p,r,\alpha} \le \dfrac{1}{4}$$
	and $$\max \left( \sqrt{\dfrac{16\, m(d_n,p,S,w)}{3n} \log \dfrac{s(d_n)}{\alpha}},\dfrac{16 \,m(d_n,p,S,w)}{3n} \log \dfrac{s(d_n)}{\alpha} \right) \le \dfrac{1}{2}.$$
	Now we obtain with probability at least $1-\alpha$, $$|\Lambda_{\mu,d_n}(x) - \Lambda_{\mu_n,d_n}(x)| \le \frac{1}{2} \Lambda_{\mu,d_n}(x).$$
	We set $$\delta_n^1 := \delta_1(d_n)
	,$$
	$$\delta_n^2:= \delta_2\left(d_n, \dfrac{1}{2}\right)$$
	and $$\delta_n = \max(\delta_n^1, \delta_n^2).$$
	Recall that we have set 
	\begin{eqnarray*}
		D_{p,S, w, \epsilon} &=& \max\bigg(2, \, \left(\dfrac{\diam(S)}{R}+1\right)^{\frac{1}{1-\epsilon}}, \left(\dfrac{2}{R} E_{p,r,\epsilon}\left(1, \dfrac{1}{2}\right)\right)^{\frac{1}{1-\epsilon}}\bigg).
	\end{eqnarray*}
	Then, for $n \ge n_0 := \dfrac{4 (D_{p,S,w, \epsilon}+1)^{p+2r+2} C_{p,r,\alpha}}{C R^{p+r}}$, $d_n = \left\lfloor \left(\dfrac{C R^{p+r}}{4C_{p,r,\alpha}} n\right)^{\frac{1}{p+2r+2}} \right\rfloor \ge D_{p,S,w,\epsilon}$.
	Hence $d_n \ge 2 > 1 $. Furthermore,
	$$\delta^1_n = \dfrac{\diam(S)}{d_n^{1-\epsilon}-1} \le \dfrac{\diam(S)}{D_{p,S,w,\epsilon}^{1-\epsilon}-1} \le \dfrac{\diam(S)}{\dfrac{\diam(S)}{R}+1-1} \le R$$
	and 
	\begin{eqnarray*}
	    \delta^2_n &=&
	\dfrac{2}{d_n^ {1-\epsilon}} E_{p,r,\epsilon}\left(d_n, \dfrac{1}{2}\right) \le \dfrac{2}{D_{p,S,w,\epsilon}^ {1-\epsilon}} E_{p,r,\epsilon}\left(d_n, \dfrac{1}{2}\right)\\
	&\le& \dfrac{2}{\dfrac{2}{R} E_{p,r,\epsilon}\left(1, \dfrac{1}{2}\right)} E_{p,r,\epsilon}\left(d_n, \dfrac{1}{2}\right) \le R,
	\end{eqnarray*}
	which implies that $\delta_n = \max(\delta_n^1, \delta_n^2) \le R$.
	
	Now, all the assumptions of Lemma \ref{lm_Hausdorff_distance_population} hold with probability at least $1-\alpha$ and we obtain the threshold $\gamma_n$ along with the estimator $S_n$ such that
	$$d_H(S, S_n) \le \delta_n$$
	and $$d_H(\partial S, \partial S_n) \le \delta_n.$$
\end{proof}
\subsection{Proof of Theorem \ref{thr_support_inference_empirical_Chf_function_symmetric_distance}} \label{appen_proofs_symmetric distance}
By combining the bounds of $S$ which have been shown in Lemma \ref{lm_bounds_of_support} with Assumption \ref{aspt_on_S} - part 2, we obtain a result concerning the Lebesgue measure of the symmetric difference for the thresholding scheme \eqref{eqn_thresholding_scheme_with_d_and_n}.
\begin{lemma} \label{lm_symmetric_distance_population}
	Under the assumptions and definitions of Lemma \ref{lm_bounds_of_support}, we suppose in addition that $S$ satisfies Assumption \ref{aspt_on_S} - part 2 with a constant $C_S > 0$. For any $d > 1$, the thresholding scheme \eqref{eqn_thresholding_scheme_with_d_and_n} satisfies that
	$$\lambda(S \triangle S_{d,n}) \le 2 C_S \delta_d + O((\delta_d)^2).$$
\end{lemma}
\begin{proof} 
	By Lemma \ref{lm_bounds_of_support} and Assumption \ref{aspt_on_S} - part 2, we have
	\begin{eqnarray*}
		\lambda(S \triangle S_{d_n}) &=& \lambda(S \setminus S_{d,n}) + \lambda(S_{d,n} \setminus S)\\
		&\le& \lambda(S \setminus S^2) + \lambda(S^1 \setminus S)\\
		&=& \lambda(\{x \in S: d(x, \partial S) < \delta_d^2\}) + \lambda(\{x \notin S: d(x, \partial S) \le \delta_d^1\})\\
		&\le& V_{\partial S}(\delta_d^2) + V_{\partial S}(\delta_d^1)\\
		&\le& C_S(\delta_d^1 + \delta_d^2) + O((\delta_d^1)^2) + O((\delta_d^2)^2)\\
		&\le& 2 C_S \delta_d + O((\delta_d)^2).
	\end{eqnarray*}	
\end{proof}
The proof of Theorem \ref{thr_support_inference_empirical_Chf_function_symmetric_distance} follows by combining Lemma \ref{lm_symmetric_distance_population} with the concentration result in Corollary \ref{coroll_concentration_result}.
\begin{proof} [Proof of Theorem \ref{thr_support_inference_empirical_Chf_function_symmetric_distance}]
	Similar to the proof of Theorem \ref{thr_support_inference_empirical_Chf_function_Hausdorff_distance}, we obtain with probability at least $1-\alpha$, $$|\Lambda_{\mu,d_n}(x) - \Lambda_{\mu_n,d_n}(x)| \le \frac{1}{2} \Lambda_{\mu,d_n}(x).$$
	If $n \ge n_1 := \dfrac{2^{p+2r+4} C_{p,r,\alpha}}{C R^{p+r}}$, then $d_n = \left\lfloor \left(\dfrac{C R^{p+r}}{4C_{p,r,\alpha}} n\right)^{\frac{1}{p+2r+2}} \right\rfloor \ge 2 > 1$.\\
	Now, all the assumptions of Lemma \ref{lm_symmetric_distance_population} hold with probability at least $1-\alpha$ and we obtain the threshold $\gamma_n$ along with the estimator $S_n$ such that
	$$\lambda(S \triangle S_n) \le 2 C_S \delta_n + O((\delta_n)^2).$$
\end{proof}

\section{Proof of Lemma \ref{lm_smooth_boundary}} \label{appen_proof_smooth_boundary_lemma}
The following proof of Lemma \ref{lm_smooth_boundary} requires some geometry concepts which can all be found in \citet{lee2000introduction} for instance.

First, we introduce the tubular neighborhood theorem in Euclidean spaces. Given a smooth submanifold $M$ of $\R^p$, we denote by $N (M)$ its normal bundle:
$N (M) = \{(x,v): x \in M, v \in N_x\},$
where $N_x$ is the normal vector space to $M$ at $x$. We define 
\begin{eqnarray*}
	E: N (M) &\longrightarrow& \R^p\\
	\quad (x,v) &\longmapsto& x+ v.
\end{eqnarray*}
A tubular neighborhood of $M$ is a neighborhood $U$ of $M$ in $\R^p$ that is the diffeomorphic image under $E$ of an open set $V \subset N (M)$ of the form
$$V = \{(x,v) \in N (M) : \|v\|_2 < \epsilon(x)\},$$
for some positive continuous function $\epsilon: M \in \R_+$. The following theorem regarding the existence of a tubular neighborhood comes from \cite{lee2000introduction}, Theorem 6.24.
\begin{theorem} [Tubular neighborhood theorem]
	Every embedded submanifold of $\R^p$ has a tubular neigborhood.
\end{theorem}

\begin{proof} [Proof of Lemma \ref{lm_smooth_boundary}]	    First, we notice that an embedded hypersurface in $\R^p$ which is a closed set of $\R^p$ is orientable (see \cite{samelson1969orientability}). Hence by assumption, we can give $\partial S$ an orientation which sides coincide with the interior and exterior of $S$ (obviously, $\partial S$ is closed).\\
Now, since $\partial S$ is an embedded hypersurface in $\R^p$, we can apply the tubular neighborhood theorem. Moreover, the fact that $\partial S$ is compact (a closed set in $S$ which is compact) implies that there exists a constant $\epsilon > 0$ and an open neighborhood $U$ of $\partial S$ such that $E$ is a diffeomorphism from $V = \{(x,v) \in N  (\partial S): \|v\|_2 < \epsilon\}$ to $U$. We set $R = \frac{\epsilon}{2}$ and we will check that $R$ satisfies the claimed hypothesis, i.e. for all $x \in S$, there exists $z \in S$ such that $x \in \overline{B}_R(z) \subset S$.\\
	For $x \in S$ fixed, we will consider two cases where $x \in U$ and $x \notin U$ separately.\\
	If $x \notin U$, we take $z =x$ and we will prove that $\overline{B}_R(z) \subset S$. We consider $y \in \text{proj}_{\partial S}(x)$, which means $y \in \partial S$ and $d(x, \partial S) = \|x-y\|_2$. Note that $x-y \in N_x$, where $N_x$ is the normal vector space of $\partial S$ at $x$. If $\|x-y\|_2 < \epsilon$, then $(y, x-y) \in V$, hence $x=y+x-y= E(y, x-y) \in U$ which is a contradiction. Thus $d(x, \partial S)=\|x-y\|_2 \ge \epsilon$, which implies $\overline{B}_R(z) \subset \overline{B}_\epsilon(z) \subset S$.\\
	Assume now that $x \in U$. First, we will prove that the projection on $\partial S$ of an arbitrary $z \in U$ is unique. Indeed, since $z \in U$ and $E|_V: V \rightarrow U$ is a diffeomorphism, there exists a unique $y_z \in \partial S$ and $v_z \in N_{y_z}$ such that $\|v_z\|_2 < \epsilon$ and $z=y_z+v_z$. If $y \in \text{proj}_{\partial S}(z)$, then $v=z-y \in N_{y}$ and $\|v\|_2 = \|z-y\|_2 = d(z, \partial S) \le \|z-y_z\|_2 = \|v_z\|_2 < \epsilon$. By bijectivity of $E|_V$, $(y,v) =(y_z,v_z)$ which is unique.\\
	Now, when $x \in U \cap S$, let $y$ be its unique projection on $\partial S$. We can write $x$ as $x = y + \lambda v_0$, where $v_0$ is the inner pointing unit vector of $\partial S$ at $x$, $0 \le \lambda < \epsilon$ by the previous argument. We set $z = x + R v_0$ and we will prove that $x \in \overline{B}_R(z) \subset S$.
	First, we have $z \in S$ since otherwise, if $z=y+Rv_0 \notin S$ when $y \in \partial S$ and $v_0$ is inner pointing, then there exists a $\tilde{R} \in (0, R) \subset (0, \epsilon)$ such that $y+ \tilde{R} v_0 \in \partial S$. Note that two couples $(y, \tilde{R} v_0)$ and $(y+\tilde{R}v_0, 0)$ in $V$ have the same image by $E$. The bijectivity of $E|_V$ implies that $\tilde{R} = 0$, which is a contradiction. Now, by the definition of $z$, we have $z \in U$ and $y$ is its unique projection on $\partial S$, hence $d(z, \partial S) = \|z-y\|_2 = R$, which implies $\overline{B}_R(z) \subset S$. Finally $x \in B_R(z)$ since $\|x-z\|_2 = |\lambda-R| \le R$, which comes from the fact that $\lambda < \epsilon$ and $R=\epsilon/2$.
	Now, we have already proved that for any $x \in S$, a ball of radius $R$ containing $x$ is included in $S$, i.e. a ball of radius $R$ rolls inside $S$. Similarly, a ball of radius $R$ rolls outside $S$. The first part of Assumption \ref{aspt_on_S} is proved.
	
	For the second part, since $\partial S$ is a smooth sub-manifold of $\R^p$ with co-dimension 1, by Weyl's formula which was for instance provided in \cite{gray2012tubes}, we have
	$$V_{\partial S}(\epsilon) = 2\epsilon \sum\limits_{c=0}^{[(p-1)/2]} \dfrac{k_{2c}(\partial S) \epsilon^{2c}}{1.3 \ldots (2c+1)} = \sum\limits_{c=0}^{[(p-1)/2]} \dfrac{2k_{2c}(\partial S) }{1.3 \ldots (2c+1)} \epsilon^{2c+1},$$
	where $k_{2c}(\partial S)$ are numbers depending on $\partial S$ but not on $\epsilon$ and $k_0$ is the volume of $\partial S$. Then we have
	$$V_{\partial S}(\epsilon) =  \epsilon \, C_S  + O(\epsilon^3),$$ where $C_S = 2 \text{Vol}(\partial S)$ only depends on $S$.
\end{proof}

\end{document}